\documentclass[a4paper,12pt]{article}
\usepackage{CJKutf8}
\usepackage{CJK}
\usepackage{amsfonts}
\usepackage{pifont}
\usepackage{bm}
\usepackage{latexsym,amsmath,amssymb,cite,amsthm}
\usepackage{color,eucal,enumerate,mathrsfs}
\usepackage[normalem]{ulem}
\usepackage{amsmath}
\usepackage{hyperref}
\usepackage[pagewise]{lineno}

\numberwithin{equation}{section}
\theoremstyle{plain}
\newtheorem{theorem}{Theorem}[section]
\newtheorem{corollary}[theorem]{Corollary}
\newtheorem{lemma}[theorem]{Lemma}
\newtheorem{proposition}[theorem]{Proposition}
\theoremstyle{definition}
\newtheorem{definition}[theorem]{Definition}
\theoremstyle{remark}
\newtheorem{assumption}[theorem]{Assumption}
\newtheorem{remark}[theorem]{Remark}

\newcommand{\geo}{\rm Geo}

\newcommand{\lmt}[2]{\mathop{\lim}_{{#1} \rightarrow {#2}} }

\newcommand{\lip}[1]{{\mathrm{lip}}({#1})}
\newcommand{\esup}[1]{{\mathrm{ess~sup}}{#1}}

\newcommand{\lmts}[2]{\mathop{\overline{\lim}}_{{#1} \rightarrow {#2}} }

\newcommand{\tr}{{\rm{tr}}}
\renewcommand{\H}{{\mathrm{Hess}}}

\newcommand{\mm}{\mathfrak m}
\newcommand{\ms}{(X,\d,\mm)}
\newcommand{\cd}{{\rm CD}(k, \infty)}

\newcommand{\rcdkn}{{\rm RCD}^*(k, N)}
\newcommand{\rcd}{{\rm RCD}(k, \infty)}

\newcommand{\N}{\mathbb{N}}

\newcommand{\R}{\mathbb{R}}
\newcommand{\g}{\mathrm{g}}
\newcommand{\Q}{\mathrm{Q}}
\newcommand{\supp}{\mathop{\rm supp}\nolimits}   %%\newcommand{\span}{\mathop{\rm span}\nolimits}   %
\newcommand{\Lip}{\mathop{\rm Lip}\nolimits}
\renewcommand{\d}{{\mathrm d}}

\newcommand{\D}{{\mathrm D}}
\newcommand{\restr}[1]{\lower3pt\hbox{$|_{#1}$}}
\newcommand{\la}{{\langle}}
\newcommand{\ra}{{\rangle}}

\newcommand{\nchi}{{\raise.3ex\hbox{$\chi$}}}

\newcommand{\lims}{\varlimsup}
\newcommand{\loc}{{\rm loc} }

\title{\Large{\bf Characterizations of monotonicity of  vector fields on metric measure spaces}
}
\begin{document}
\author{ Bang-Xian Han \thanks
{  University of Bonn,  Institute for  Applied Mathematics,
han@iam.uni-bonn.de} 
}

\date{\today}
\maketitle

\begin{abstract}
We characterize the convexity of functions and the monotonicity of  vector fields on  metric measure spaces with Riemannian Ricci curvature bounded from below.  Our result offers a new approach to deal with some rigidity theorems such as ``splitting theorem" and ``volume cone implies metric cone theorem" in non-smooth context. 
\end{abstract}



\tableofcontents

%%%%%%%%%%%%%%%%%%%%%%%%%%%%
% 前言
%%%%%%%%%%%%%%%%%%%%%%%%%%%%

\section{Introduction}
In the past twenty years, the displacement convexity of functionals on Wasserstein space has been deeply studied,  and it has applications  in many fields such as differential equation theory, probability theory, differential and metric geometry (see   \cite{AGS-G} and \cite{V-O} for an overview of related theories). 

\bigskip

One of the most  interesting functionals is the Boltzmann entropy. Let  $(M, \g, V_\g)$ be a Riemannian manifold.  The Boltzmann entropy ${\rm Ent}_{V_\g}(\cdot)$ is defined by
\begin{equation*}
{\rm Ent}_{V_\g}(\mu):=
\left \{\begin{array}{ll}
 \int \rho \ln \rho \,\d{V_\g}  ~~~\text{if}~~\mu=\rho \,V_\g,\\
\infty  ~~~~~~~~~~~~~~~\text{otherwise}
\end{array}\right.
\end{equation*}
It is known from  \cite{SVR-T} that  the  convexity of ${\rm Ent}_{V_\g}(\cdot)$ in Wasserstein space
characterizes the lower Ricci curvature bound of $M$. It is proved by Erbar  in \cite{E-H} that the gradient flow of ${\rm Ent}_{V_\g}$ in Wasserstein space can be identified  with the heat flow in the following sense: let $\bar{\mathcal{H}}_t(fV_\g)$ be the Wasserstein gradient flow of ${\rm Ent}_{V_\g}$ starting from $fV_\g \in \mathcal{P}_2(M)$, $\mathcal{H}_t(f)$ be the solution of the heat equation with initial datum $f\in L^2(M, V_\g)$,  then $\mathcal{H}_t(f)\,V_\g=\bar{\mathcal{H}}_t(fV_\g)$.

Moreover, we have  the following well-known theorem. 

\begin{theorem} [Von Renesse-Sturm \cite{SVR-T}, Erbar  \cite{E-H}] \label{th-intro-1}
 Let $(M, \g, V_\g)$ be a Riemannian manifold. Then the following characterizations are equivalent.

\begin{itemize}
\item [1)] The Ricci curvature of $M$ is uniformly bounded from below by a constant $K\in \R$.
\item [2)] The entropy  ${\rm Ent}_{V_\g}(\cdot)$ is $K$-convex  in Wasserstein space.
\item [3)] For any  probability  measure $\mu$, there exists a unique  ${\rm EVI}_K$-gradient flow of ${\rm Ent}_{V_\g}$ in Wasserstein space starting from $\mu$.
\item [4)] The exponential contraction of the heat flows in Wasserstein distance
\[
W_2\big (\mu^1_t, \mu^2_t\big )\leq e^{-Kt}W_2(\mu^1_0, \mu^2_0), ~~\forall t>0
\]
holds for any two heat flows  $\mu_t^i:=\mathcal{H}_t (f^i)V_\g$, $i=1,2$.

\item [5)] There exits a  heat kernel $\rho_t(x, \d z)V_\g(z)=\bar{\mathcal{H}}_t(\delta_x)$ for any $x\in X$,  such that  the exponential contraction of heat kernels in Wasserstein distance
\[
W_2\big(\rho_t(x, \d z)V_\g(z), \rho_t(y, \d z)V_\g(z)\big) \leq e^{-Kt} \d(x, y)
\]
holds for any $x, y \in X$ and $t>0$.
\item [6)]  The gradient estimate of heat flow
\[
|\D \mathcal{H}_t(f) |^2(x) \leq e^{-2Kt}\mathcal{H}_t(|\D f|^2)(x), ~~V_\g-\text{a.e.}~ x\in X
\]
holds for any $f\in W^{1,2}(M)$.
\end{itemize}
\end{theorem}

The   notion of synthetic  curvature-dimension condition of (non-smooth) metric measure spaces is proposed by  Lott-Villani in \cite{Lott-Villani09}  and Sturm in \cite{S-O1, S-O2}. The characterization $2)$ in Theorem \ref{th-intro-1} is adopted as a definition of synthetic lower Ricci curvature bound for an abstract metric measure space.
Later on,  the curvature-dimension condition is refined  by Ambrosio-Gigli-Savar\'e (see \cite{AGS-M} and \cite{G-O}), which we call Riemannian curvature-dimension condition or ${\rm RCD}$ condition for short. It is known that the family of $\rcd$ spaces  includes  weighted Riemannian 
manifolds satisfying  curvature-dimension condition \`a la Bakry-\'Emery,  as well as their  measured Gromov-Hausdorff limits, and Alexandrov spaces with lower curvature bound.

 In  ${\rm RCD}$ setting,   heat flow is  understood as the $L^2$-gradient flow of Cheeger energy.  All those chatacterizations on manifolds  are known to be valid in  appropriate weak sense on metric measure spaces (see \cite{AGS-B, AGS-C, AGS-M}). Furthermore, more entropy-like (internal energy) functionals have been studied in \cite{AMS-N} and \cite{EKS-O}, which can be used to characterize  $\rcdkn$ condition.

\bigskip

Besides the Boltzmann entropy (and other internal energy functionals), another important example is  the potential energy
\[
U(\cdot): {\mathcal{P}_2}(X) \ni \mu \mapsto \int_X u(x)\,\d \mu(x),
\]
where $u$ is a lower semicontinuous  function  whose negative part has squared-distance growth. When $X$ is  $\R^n$ or  a general Hilbert space, it is known from  \cite{AGS-G} that  the convexity of $u$ can be characterized using $U(\cdot)$ and its gradient flow.
Then we would like to characterize the convexity of $U(\cdot)$ in the setting of non-smooth metric measure spaces.  In this direction, several results have been obatined by Sturm, Ketterer etc., in \cite{S-G, K-O, S-N, GKKO-R}.  However,  there are still some missing components in the expected characterization theorem.  So the  first aim of the current work is to fill these gaps. 
 
 On the other hand, for  Ricci-limit spaces which are measured Gromov–Hausdorff limits of Riemannian manifolds with
Ricci curvature uniformly bounded from below, there are  two important (almost) rigidity theorems  ``(almost) splitting theorem" and `` (almost) volume cone implies (almost) metric cone theorem"  (see \cite{CC-L, CC-O1, CC-O2,CC-O3} by Cheeger and Colding).
In the proofs of these rigidity theorems on Ricci-limtit spaces,   the analysis on some special $K$-convex functions play key roles. More precisely, the existence of some $K$-convex functions implies warped product structure of the spaces.  For example, in  ``volume cone implies metric cone theorem" (cf. \cite{CC-L, DPG-F}), the target function is the squared distance function $u:=\frac12\d^2(\cdot, {\rm O})$ where ${\rm O}$ is a fixed point. In this case $\H_u={\rm Id}_N$, so $u$ is a ``$N$-convex function".  In  ``splitting theorem" (cf. \cite{Cheeger-Gromoll-splitting, G-S}), the target function is the Busemann function associated to a  line which is harmonic, so  it can be regarded as a ``0-convex function". In  $\rcd$ and $\rcdkn$  metric measure spaces, it is still interesting to reprove these rigidity results (cf. \cite{DPG-F, G-S})  by studying the corresponding $K$-convex functions. 
 
As we mentioned above, in \cite{K-O, S-G, LierlSturm2018}  the authors have proved some results concerning the $K$-convex functions on $\rcdkn$ metric measure spaces. However these results are still insufficient to study the rigidity theorems in practice, since the prerequisite  seems to be too restrictive.   This encourages us to study $K$-convex functions  in  ${\rm RCD}$ setting  with fewer assumptions.

\bigskip

Before introducing the main results, we should clarify the relationship between the Wasserstein gradient flow of $U(\cdot)$ and the flow generated by the non-smooth vector field  $\nabla u$, as we identify the heat flow and the gradient flow of entropy before.

 In \cite{AT-W},  Ambrosio and Trevisan extend the famous Di Perna-Lions theory to $\rcd$ metric measure spaces. They prove that the  continuity equation 
\begin{equation}\label{eq:continsym}
\partial_t\mu_t+\nabla\cdot(-\mu_t\nabla u)=0,
\end{equation} 
is well posed under some assumptions on the Sobolev regularity of $u$.  They prove the existence and uniqueness of the solution to \eqref{eq:continsym} for any initial condition $\mu_0 \in \mathcal{P}(X)$ with $\mu_0 \leq C_0\mm$.  They also show the existence and uniqueness of the regular Lagrangian flow $(F_t)_{t\in [0, T)}: X \mapsto X $ such that  $\mu_t=(F_t)_\sharp \mu_0 \leq C_1\mm$. 

In \cite{GH-C},  Gigli and the author  study   the absolutely continuous curves in Wasserstein space through its corresponding continuity equation.  It is proved  that  $(\mu_t)$ solves \eqref{eq:continsym} if and only if it is a gradient flow of $U: \mu \mapsto \int u\,\d \mu$ in Wasserstein space.  Formally speaking,  $(\mu_t)$ is the gradient flow of $U$ if and only if the velocity field of its corresponding continuity equation is $-\nabla u$.

The first main result in this paper is  the following characterization theorem.

\begin{theorem}[Theorem \ref{mainth-0} and Theorem \ref{mainth-1}]\label{theorem-1}
Let $\ms$ be a $\rcd$ space, $u$ be a scalar function with appropriate  regularities.
Then the following characterizations are equivalent.
\begin{itemize} 
\item [1)] $u$ is   {infinitesimally $K$-convex}:  $\H_u(\nabla f, \nabla f) \geq K|\D f|^2$ $\mm$-a.e. for all $f\in W^{1,2}$, where $\H_u(\cdot, \cdot)$ is the Hessian of $u$.
\item [2)] $u$ is {\sl weakly $K$-convex}, i.e. $U(\cdot)$ is $K$-displacement convex in Wasserstein space.
\item [3)] $\nabla u$ is {\sl $K$-monotone} in the sense that  
\[
\int \la\nabla u, \nabla \varphi \ra\,\d \mu^1+\int \la \nabla u, \nabla \varphi^c \ra\,\d \mu^2 \geq K W^2_2(\mu^1,\mu^2)
\]
for any $\mu^1, \mu^2 \in \mathcal{P}_2$ with bounded densities and bounded supports, where $(\varphi, \varphi^c)$ are Kantorovich potentials associated to $(\mu^1, \mu^2)$.

\item [4)] The exponential contraction in Wasserstein distance
\[
W_2(\mu^1_t, \mu^2_t)\leq e^{-Kt}W_2(\mu^1_0, \mu^2_0), ~~\forall t>0
\]
holds for any two solutions $(\mu_t^1), (\mu_t^2)$ to the continuity equation \eqref{eq:continsym}.
\item [5)] The regular Lagrangian flow $(F_t) $ associated to $-\nabla u$ has a unique continuous representation, so that $F_t(x)$ can be determined at any $x\in X$. Furthermore, the exponential contraction
\[
\d(F_t(x), F_t(y)) \leq e^{-Kt} \d(x, y)
\]
holds for any $x, y \in X$ and $t>0$. 
\item [6)] For any $f\in W^{1,2}\ms$, we have $f\circ F_t \in W^{1,2}$ and
\[
|\D (f\circ F_t) |(x) \leq e^{-Kt}|\D f|\circ F_t(x), ~~\mm-\text{a.e.} ~x\in X.
\]
\end{itemize}
\end{theorem}
 
 \bigskip
 
 We will divide the  proof of the  theorem above into two parts. The first part is {Theorem} \ref{mainth-0}, which deals with the equivalence of $1)$ and $2)$. It has been proved (in e.g. \cite{K-O, GKKO-R, LierlSturm2018}) when  $u$ is a test function (see Section \ref{section-sobolev} for the definition). However, in some potential applications, such as `` splitting theorem" and ``volume cone implies metric cone theorem", the  target functions only have lower differentiability and integrability. In Theorem \ref{mainth-0}, using some recent results on second order differential calculus on metric measure spaces, we can just assume that $u \in W^{2,2}_\loc$,  locally bounded and $u(x) \geq -a-b \d^2(x, x_0)$ for some $a, b \in \R^+, x_0\in X$.    We remark that the regularity of convex functions on metric measure space is still unclear, especially when the space is not locally compact. So  regularity assumptions are still necessary, and we conjecture that the assumption in this article is almost minimal.

 The second part of the proof is {Theorem} \ref{mainth-1}, in which  we prove the equivalence of $2)-6)$. The well-posedness  of  this theorem requires the existence and uniqueness theory of regular Lagrangian flow on metric measure spaces, which is studied by Ambrosio-Trevisan \cite{AT-W}. For potential applications of this theorem, we also extend Ambrosio-Trevisan's result to a lager class of vector fields in Proposition \ref{prop:rlf-2}.

 \bigskip
 
 Consequently, we will see in {Theorem} \ref{mainth-2} that the $K$-monotonicity of a (possibly) non-symmetric vector field ${\bf b}$ can be characterized in similar ways as $ 3), 4), 5), 6)$ in Theorem \ref{theorem-1}.   To our knowledge,  these  equivalent descriptions  are  new even on Riemannian manifolds and Riemannian limit spaces. Due to lack of second order differentiation formula,  and  low regularity of the  vector field, the usual argument in smooth setting fails to work under $\rcd$ condition (see also Remark \ref{remark:locallip}).  
 
\noindent
 {\bf Theorem \ref{mainth-2}}
 {\it Let $M:=\ms$ be a $\rcd$  space, ${\bf b} \in L^2_\loc(TM)$. We assume that  there exits a unique  regular Lagrangian flow associated to $-{\bf b}$,   which is denoted by $(F_t)$. Then the following descriptions are equivalent.

\begin{itemize}
\item [1)] ${\bf b}$ is $K$-monotone.
\item [2)] The exponential contraction in Wasserstein distance
\[
W_2(\mu^1_t, \mu^2_t)\leq e^{-Kt}W_2(\mu^1_0, \mu^2_0), ~~\forall t>0
\]
holds for any two curves  $(\mu_t^1), (\mu_t^2)$ whose velocity fields are $-{\bf b}$.

\item [3)] The regular Lagrangian flow $(F_t) $ associated to   $-{\bf b}$  has a unique continuous representation, so that $F_t(x)$ is well-defined at any $x\in X$. Furthermore, the exponential contraction
\[
\d(F_t(x), F_t(y)) \leq e^{-Kt} \d(x, y)
\]
holds for any $x, y \in X$ and $t>0$.
\item [4)] For any $f\in W^{1,2}\ms$, we have $f\circ F_t \in W^{1,2}$ and
\[
|\D (f\circ F_t) |(x) \leq e^{-Kt}|\D f|\circ F_t(x), ~~\mm-\text{a.e.} ~x\in X.
\]

\end{itemize}
}

\bigskip
At last, we summarize the highlights and main innovations in this paper.
\begin{itemize}
\item [a)]   We improve  some results  concerning  $K$-convex function (cf. \cite{K-O,S-G, LierlSturm2018}), and continuity equation on metric measure spaces (cf. \cite{AT-W, GH-C}).
\item [b)] We  characterize the $K$-convex  functions on $\rcd$ spaces in several equivalent ways.
\item [c)] We propose some new concepts to characterize the monotonicity of non-smooth vector fields, and characterize the  $K$-monotone vector fields  in several equivalent ways.
\item [d)] The characterization theorems improve our understanding of  $K$-convex functions and $K$-monotone vectors on Riemannian manifolds.
\item [e)] We find a  new  approach to study some rigidity theorems on  spaces with lower Ricci curvature bound.
\end{itemize}

\bigskip

The paper is organized as follows.
In section 2 we review some basic results on optimal transport, Sobolev spaces and $L^\infty$-modules on metric measure spaces, and continuity equation on metric measure spaces studied in \cite{AT-W} and \cite{GH-C}. In section 3, we prove our main theorems which characterize the $K$-convex functions and $K$-monotone vector fields on metric measure spaces.   In the last section,  we apply our characterization theorem to prove two  results,  which are key steps in the proofs of ``splitting theorems" and ``from volume cone to metric cone  theorem".

%%%%%%%%%%%%%%%%%%%%%%%%%%%%
% 预备知识
%%%%%%%%%%%%%%%%%%%%%%%%%%%%

\section{Preliminaries} 
\subsection{Metric measure space and optimal transport}
We recall some basic  results concerning analysis on metric spaces and optimal transport theory. More detailed discussions can be found  in   \cite{AG-U, AGS-G} and  \cite{V-O}.
Basic assumptions on the metric measure spaces in this paper are listed below.
\begin{assumption}\label{assumption}
The metric measure space $M:=\ms$ satisfies:
\begin{itemize}
\item [i)] $(X,\d)$ is a complete and separable geodesic  metric space,
\item [ii)]$\supp{\mm}=X$,
\item [iii)]$\mm$ is a  $\d$-Borel  measure  and takes  finite value on bounded sets,
\item [iv)]$\ms$ has exponential volume growth: $\int e^{-\lambda\d^2(x, x_0)}\,\d \mm(x)<\infty$ for some $\lambda>0, x_0 \in X$.
\end{itemize}

\end{assumption}

The local Lipschitz constant $\lip f:X\to[0,\infty]$  of a function $f$ is defined by
\[
\lip f(x):=
\left\{ \begin{array}{ll}
\lims_{y\to x}\frac{|f(y)-f(x)|}{\d(x,y)},~~~~x~~\text{is not isolated}\\
0,~~~~~~~~~~~~~~~~~~~~~~~~\text{otherwise.}
\end{array}\right.
\]

The space of continuous curves on $[0,1]$ with values in $X$ is denoted by ${\rm C}([0,1],X)$ and equipped with the uniform distance.  Its  subspace consisting of constant speed geodesics is denoted by ${\rm Geo}(X)$. For $t\in[0,1]$ we define the ``evaluation map" $e_t:{\rm C}([0,1],X)\mapsto X$ by
\[
e_t(\gamma):=\gamma_t,\qquad\forall \gamma\in {\rm C}([0,1],X).
\]

A curve $\gamma:[0,1]\to X$ is called absolutely continuous if there exists $f\in L^1([0,1])$ such that
\begin{equation}
\label{eq:defac}
\d(\gamma_s,\gamma_t) \leq \int_t^s f(r) \,\d r,\qquad\forall t,s \in [0,1], \ t<s.
\end{equation}
For an absolutely continuous curve $\gamma$,  it can be proved that the limit $\lim_{h\to 0}\frac{\d(\gamma_{t+h},\gamma_t)}{|h|}$ exists for a.e. $t$ and thus defines a function, called metric speed and denoted by $|\dot\gamma_t|$, which is in $L^1([0,1])$. If $|\dot\gamma_t| \in L^2([0,1])$, we say that the curve is 2-absolutely continuous and denote the set of 2-absolutely continuous curves by ${\rm AC}^2([0,1],X)$.

\bigskip

We denote by $\mathcal{P}(X)$  the space of Borel probability measures on $X$, and denote by $\mathcal{P}_2(X)$ the space of probability measures such that $\mu \in \mathcal{P}_2(X)$ if $\mu \in \mathcal{P}(X)$
and $\int \d^2(x,x_0)\, \d\mu(x) < +\infty$ for some  $x_0 \in X$. We equip $\mathcal{P}_2(X)$ with the $L^2$-transportation distance $W_2$ (2-Wasserstein distance) defined by:
\begin{equation}
\label{eq:defw2}
W_2^2(\mu,\nu):=\inf\int \d^2(x,y)\,\d\pi(x,y),
\end{equation}
where the infimum is taken among all $\pi\in\mathcal{P}(X^2)$ with marginals $\mu, \nu$.

The measures which attain the infimum are called optimal transport plans and  denoted by ${\rm Opt}(\mu, \nu)$.
Given $\varphi:X\to\R\cup\{-\infty\}$, which is  not identically $-\infty$,  the $c$-transform  $\varphi^c:X\to\R\cup\{-\infty\}$ is defined by
\[
\varphi^c(y):=\inf_{x\in X}\frac{\d^2(x,y)}2-\varphi(x).
\]
A function $\varphi$ is said to be $c$-concave if it is not identically $-\infty$ and $\varphi=\psi^c$ for some $\psi:X\to\R\cup\{-\infty\}$. 
For $\mu,\nu\in\mathcal{P}_2(X)$,  it is known that $W_2^2(\mu,\nu)$ can be obtained as maximization of the dual problem 
\begin{equation}
\label{eq:dual}
\frac12W_2^2(\mu,\nu)=\sup\int \varphi\,\d\mu+\int\varphi^c\,\d\nu,
\end{equation}
where the supremum is taken among all $c$-concave functions $\varphi$. The supremum can be achieved and any maximizing  $\varphi$ is called Kantorovich potential from $\mu$ to $\nu$. For any Kantorovich potential we have in particular $\varphi\in L^1(\mu)$ and $\varphi^c\in L^1(\nu)$. Equivalently, the supremum in \eqref{eq:dual} can be taken among all $\varphi:X\to \R$ with Lipschitz and bounded.

\bigskip

Absolutely continuous curves in $(\mathcal{P}_2, {W}_2)$  can be characterized  by the following proposition.

\begin{proposition}[Superposition principle, \cite{L-C}]\label{lift}
Let $(X,\d)$ be a complete and separable metric space, and $(\mu_t)_{t \in [0,1]} \in  {\rm AC}^2([0,1], \mathcal{P}_2)$. Then there exists a measure $\Pi \in \mathcal{P}({\rm C}([0,1],X))$ concentrated on
${\rm AC}^2([0,1],X)$ such that:
\begin{eqnarray*}
(e_t)_\sharp \pi &=& \mu_t,~~~~~~\forall t \in [0,1]\\
\int |\dot{\gamma}_t|^2 \,\d \pi(\gamma) &=&|\dot{\mu}_t|^2, ~~~~~a.e.~t.
\end{eqnarray*}
Such a measure $\Pi$  is called a lifting of $(\mu_t)$.
\end{proposition}

\subsection{Sobolev spaces and tangent modules}\label{section-sobolev}

In this article, we adopt the definition of Sobolev space $W^{1,2}(M)$ as in \cite{AGS-C}. We say that $f\in L^2(X, \mm)$ is a Sobolev function in $W^{1,2}(M)$ if there exists a sequence of Lipschitz  functions $\{f_n\} \subset L^2$,  such that $f_n \to f$ and $\lip{f_n} \to G$ in $L^2$ for some $G \in L^2(X, \mm)$. It is known that there exists a minimal function $G$ in $\mm$-a.e. sense. We call this minimal $G$ the minimal weak  upper gradient (or weak gradient for simplicity) of  $f$, and denote it by $|\D f|$.
It is known that  locality holds for weak gradients, i.e. $|\D f|=|\D g|$ $\mm$-a.e. on the set $\{x\in X:f(x)=g(x)\}$. Similarly, we define local Sobolev space $W ^{1,2}_{\rm loc} (M)$ which consists of functions $f\in L^2_\loc$ such that for any open set $\Omega$ with bounded closure,   $f\restr{\Omega} \in W^{1,2}(\Omega)$.

Besides locality, we have the lower semi-continuity: if $(f_n)_n \subset W^{1,2}$ converges to some $f\in L^2$ in $\mm$-a.e. sense and such that $(|\D f_n|)_n$ is bounded in $L^2(X,\mm)$, then $f\in W^{1,2}\ms$ and
\[
|\D f|\leq G,\qquad\mm\text{-a.e.}
\]
for every $L^2$-weak limit $G$ of some subsequence of $(|\D f_n|)_n$.

We equip $W^{1,2}\ms$ with the norm
\[
\|f\|^2_{W^{1,2}\ms}:=\|f\|^2_{L^2(X,\mm)}+\||\D f|\|^2_{L^2(X,\mm)}.
\]
It is known that $W^{1,2}\ms$ is a Banach space, but not necessary a Hilbert space. We say that $\ms$ is an infinitesimally Hilbertian space if $W^{1,2}\ms$ is a Hilbert space.

On an infinitesimally Hilbertian space $M$, we define a pointwise bilinear map $\Gamma(\cdot, \cdot)$ by
 \[
[W^{1,2}(M)]^2 \ni (f, g) \mapsto \Gamma( f,   g):= \frac14 \Big{(}|\D (f+g)|^2-|\D (f-g)|^2\Big{)}.
\]
We have the following Leibniz rule (see Proposition 3.17  \cite{G-O} for a proof):
\[
\Gamma(fg, h)=f\Gamma(g,h)+g\Gamma(f,h)
\]
for any $f, g, h \in W^{1,2}\cap L^\infty$.

\begin{definition}
[Measure valued Laplacian, \cite{G-O, G-N}]
The space ${\rm D}({\bf \Delta}) \subset   L^2_{\rm loc} (M)$ is the set of $f \in  W ^{1,2}_{\rm loc} (M)$ such that there is a Radon measure ${\bf \mu}$ satisfying
\[
\int h \,\d{\bf \mu}= -\int \Gamma( h, f ) \, \d  \mm~~~ \forall h: M \mapsto  \R, ~~ \text{Lipschitz with bounded support}.
\]
In this case the measure $\mu$ is unique and we denote it by ${\bf \Delta} f$. If ${\bf \Delta} f \ll m$, we denote its density by $\Delta f$. 
\end{definition}

\begin{remark}
We do not assume that ${\bf \Delta} f$ has bounded total variation in this paper, so $\Delta f$ is not necessarily $L^1$-integrable, but only locally integrable. 
\end{remark}

If $\Delta f\in L^2$, by definition  we know 
\[
\int \varphi \Delta f\,\d\mm= -\int \Gamma(\varphi, f ) \, \d  \mm
\]
for any $\varphi \in W^{1,2}$.

Let $(f_n)_{n=0}^\infty \subset {\rm D}({\bf \Delta})$.  We say that  $(f_n)$ converges to $f_\infty$ in ${\rm D}({\bf \Delta})$ if $f_n \to f$ in $W^{1,2}$ and $\Delta f_n \to \Delta f_\infty$ in $L^2$.

\bigskip

Next we will characterize the curvature-dimension conditions $\rcd$ and $\rcdkn$ using non-smooth Bakry-\'Emery theory. We firstly recall that a space is  $\rcd$ (or $\rcdkn$) if it is a  ${\rm CD}(K, \infty)$ (or ${\rm CD}^*(K, N)$)  space in the sense of  Lott-Sturm-Villani \cite{Lott-Villani09, S-O1, S-O2} (or Bacher-Sturm  \cite{BS-L}),  equipped with an infinitesimally Hilbertian Sobolev space. For more details, see \cite{AGS-M} and \cite{AGMR-R}.

We define  ${\rm TestF}(M) \subset W^{1,2}(M)$, the set of test functions  as
\[
{\rm TestF}(M):= \Big\{f \in {\rm D} ({\bf \Delta}) \cap L^\infty: f\in W^{1,2}\cap \Lip~~ {\rm and}~~~ \Delta f  \in W^{1,2}\cap L^\infty \Big \}.
\] It is known that ${\rm TestF}(M)$ is dense in $W^{1,2}(M)$ when $M$ is ${\rm RCD}$.

\begin{lemma}\label{lemma:boundedsupport}
We denote  by $ {\rm TestF}_{\rm bs}(M)\subset {\rm TestF}(M)$ the space of test functions with bounded support, then $ {\rm TestF}_{\rm bs}(M)$ is dense in ${\rm TestF}(M)$ with respect to both $W^{1,2}$ and ${\rm D}({\bf \Delta })$ topology.
\end{lemma}
\begin{proof}
Let $\chi_n \in {\rm TestF}$, $n\in \N$ be cut-off functions (cf. Lemma 6.7, \cite{AMS-O}) such that
\begin{itemize}
\item [a)] $0\leq \chi_n \leq 1$, $\chi_n$ supports on $B_{3n}(x_0)$ and $\chi_n=1$ on $B_{n}(x_0)$,
\item [b)] $\Lip (\chi_n) \leq \frac 1n$,
\item [c)] $\Delta \chi_n \in L^\infty$ uniformly in $n$ and $|\D \chi_n|^2 \in W^{1,2}$.
\end{itemize}

For any $f\in {\rm TestF}$ we define $f_n:= \chi_n f$. Then we have $\nabla f_n=f\nabla \chi_n +\chi_n \nabla f$, $\Delta f_n=
f\Delta \chi_n +\chi_n \Delta f+2\la \nabla f, \nabla \chi_n \ra$. Hence  we know $f_n \in {\rm TestF}_{\rm bs}$,  $f_n \to f$ in $W^{1,2}$ and
$\Delta f_n \to \Delta f$ in $L^2$.  So  $ {\rm TestF}_{\rm bs}$ is dense in  ${\rm TestF}$ with respect to both $W^{1,2}$ and ${\rm D}({\bf \Delta })$ topology.
\end{proof}

\bigskip

 Let $f,g \in {\rm TestF}(M)$. We know (from \cite{S-S}) that $\Gamma(f,g) \in {\rm D}({\bf \Delta})$, so we can define  the measure ${\bf \Gamma}_2(f,g)$  by
\[
{\bf \Gamma}_2(f,g)=\frac12 {\bf \Delta}\Gamma( f,  g) -\frac12 \big{(}\Gamma( f,  \Delta g)+\Gamma( g,  \Delta f)\big{)}\, \mm,
\]
and we put ${\bf \Gamma}_2(f):={\bf \Gamma}_2(f,f)$. Then we have the following Bochner inequality on metric measure space, which can be regarded as a variant definition of $\rcd$ and $\rcdkn$ conditions.

Firstly we recall the Sobolev-to-Lipschitz property, which is an important  prerequisite in   Bakry-\'Emery theory,  see \cite{AGS-B} and \cite{GH-S}  for more discussion about this property.

\begin{definition}[Sobolev-to-Lipschitz property]\label{def:stl}
We say that   a metric measure space $(X ,\d, \mm)$ has  Sobolev-to-Lipschitz property if for any function $f\in W^{1,2}(X)$  with $|\D f| \in L^\infty(X)$, we can find a Lipschitz function  $\tilde{f}$ such that $f=\tilde{f}$ $\mm$-a.e. and  $\Lip (\tilde f)=\esup ~{|\D f|}$.
\end{definition}

\begin{proposition} [Bakry-\'Emery condition, \cite{AGS-B, AGS-M} and \cite{EKS-O}]\label{becondition}
Let $M=\ms$ be a $\rcdkn$ space with $k \in \R$ and $N \in [1, \infty]$. Then 
\[
{\bf \Gamma}_2(f) \geq \Big {(} k \Gamma( f)+ \frac1N (\Delta f)^2 \Big{)}\,\mm
\]
for any $f \in {\rm TestF}(M)$.

Conversely, let $M=\ms$ be  an infinitesimally Hilbertian space satisfying Sobolev-to-Lipschitz property, fulfil the Assumption \ref{assumption}. Then it is a $\rcdkn$ space with $k \in \R$ and $N \in [1, \infty]$ if
\[
\frac12 \int |\D f|^2 \Delta \varphi\,\d\mm-\int \la \nabla f, \nabla \Delta f\ra \varphi\,\d\mm \geq k\int |\D f|^2\varphi\,\d\mm+\frac1N \int (\Delta f)^2\varphi\,\d\mm
\]
for any $\varphi \in {\rm D}_{L^\infty}(\Delta)$ and $f \in {\rm D}_{W^{1,2}}(\Delta)$, where
$$
{\rm D}_{L^\infty}(\Delta):=\Big \{ \varphi: \Delta \varphi \in L^2 \cap L^\infty, \varphi\in W^{1,2}\cap L^\infty \Big \},
$$
and
$$
{\rm D}_{W^{1,2}}(\Delta):=\Big \{ \varphi: \varphi \in W^{1,2},  \Delta\varphi\in W^{1,2} \Big \}.
$$
\end{proposition}

\bigskip

Next, we will review the concept ``(co)tangent vector field" in non-smooth setting. Detailed definition and basic properties of  $L^\infty$-module can be found in \cite{G-N}. 

The cotangent module of $M$,  which is a $L^2$-normed $L^\infty$-module denoted by $(L^2(T^*M), \| \cdot\|, |\cdot|)$. For any $f\in W^{1,2}(M)$, there exists $\d f \in L^2(T^*M)$ such that $|\d f|=|\D f|$ and $\| \d f\|=\| |\D f| \|_{L^2}$.

We  define the tangent module $L^0(TM)$ as  $\mathrm{Hom}_{L^\infty(M)}(L^2(T^*M), L^0(M))$, i.e. $T \in L^0(T M)$ if it is a  linear map from $L^2(T^*M)$  to $L^0(M)$ as Banach spaces with the $L^\infty$-homogeneity:
\[
T(f v)=f  T(v), ~~\forall v \in L^2(T^*M),~~ f \in L^\infty(M),
\]
and continuity:
\[
T(v)\leq G |v|~~\mm-\text{a.e.},~~~\forall~v\in L^2(T^*M)
\]
for some $G\in L^0$.  The smallest function $G$ satisfying this property will be denoted by $|T|$.
For example,  for any $f\in W^{1,2}_\loc(M)$,  there exists an element in $L^0(TM)$ denote by $\nabla f$, such that $\nabla f(\d g)= \Gamma(f, g)\leq |\D f||\D g| $ for any $g\in W^{1,2}$. So $|\nabla f|=|\D f|\in L^2_\loc$.

We define $L^2(TM)$ as the space consisting of vectors $T\in  L^0(TM)$ such that  $|T| \in L^2(M)$.  It can be seen that $L^2(TM)$ has a natural $L^2$-normed $L^\infty(M)$-module structure, and it  is isometric to $L^2(T^*M)$ both as a module and  a Hilbert space. We denote the corresponding element of $\d f$ in $L^2(TM)$ by $\nabla f$ and call it the gradient of $f$. The natural pointwise norm on $L^2(TM)$ (we also denote it by $| \cdot |$) satisfies $|\nabla f|=|\d f|=|\D f|$. It is also known that  $\{\sum_{i \in I} a_i  \nabla f_i: |I|<\infty,  a_i \in L^\infty(M), f_i \in W^{1,2} \}$ is dense in $L^2(T M)$. Alternatively,  we  define a pointwise inner product $\la \cdot, \cdot \ra: [L^2(T^*M)]^2 \mapsto L^1(M)$ by
 \[
 \la \d f, \d  g \ra:=\Gamma(f,g)= \frac14 \Big{(}|\D (f+g)|^2-|\D (f-g)|^2\Big{)}.
\]  Then we can define the gradient $\nabla g$ as the unique element in $L^2(TM)$ such that $\nabla g (\d f):=\la \d f, \d g \ra$,  $\mm$-a.e. for every $f \in W^{1,2}(M)$. Therefore, $L^2(TM)$ inherits a pointwise inner product from $L^2(T^*M)$ and we still use $\la \cdot, \cdot \ra$ to denote it.  We  define  $L^2_\loc(TM)$ as  $\{{\bf b} \in L^0(TM): |{\bf b}|\in L^2_\loc(M)\}$.  It can be seen  that  $L^2_\loc(TM)$ inherits a pointwise inner  product from $L^2(T M)$.

\bigskip

Next we review the definitions and basic properties of the Sobolev spaces $W^{2,2}(M)$ and $W^{1,2}_C(TM)$. 
It is proved in Lemma 3.2 of \cite{S-S} that $\la \nabla f, \nabla g \ra \in {\rm D}({\bf \Delta})$ for any $f, g \in {\rm TestF}(M)$. Therefore we can define the Hessian of $f \in {\rm TestF}(M)$, which is a bilinear map  $\H_f: \{ \nabla g: g \in {\rm TestF}(M)\}^2 \mapsto L^0(M)$   by
\begin{equation}\label{eq:hessian}
2\H_f(\nabla g,\nabla h)=\la \nabla g, \nabla \la \nabla f, \nabla h \ra \ra +\la \nabla h, \nabla \la \nabla f, \nabla g \ra \ra-\la \nabla f, \nabla \la \nabla g, \nabla h \ra \ra
\end{equation}
 for any $g, h \in {\rm TestF}(M)$. It is known that $\H_f(\cdot, \cdot)$ can be extended to a continuous symmetric $L^\infty(M)$-bilinear map on $[L^2(TM)]^2$  with values in $L^0(M)$.

\begin{definition}
[Distributional divergence, \cite{AT-W, G-N}]
The domain of divergence ${\rm D}({\rm div}) \subset  L^2_\loc(TM)$ is the space  consisting of  $X \in  L ^{2}_\loc (TM)$ for which  there exists a function $f\in L_{\rm loc}^2(X, \mm)$ such that
\[
\int fg \,\d \mm= -\int \la X, \nabla g \ra \,  \d \mm, ~~~~~\forall g ~~~ \text{Lipschitz with bounded support}.
\]
In this case, we call (the unique)  $f$ the divergence of $X$ and denote it by ${\rm div}X$.
\end{definition}
It can be seen  (cf. section 2.3.3 in \cite{G-N}) that $ {\rm div}( \varphi X):=\la \nabla \varphi, X \ra+f {\rm div} X$ for $ \varphi \in \Lip (M)\cap L^\infty$ and $X\in  {\rm D}({\rm div})$.

\bigskip

We denote the pointwise scalar product of two tensors $X, Y \in L^2(TM) \otimes L^2(TM)$ by $X:Y$, and denote by $|X|_{\rm HS}^2:=\sqrt{X:X}$  the Hilbert-Schmidt norm of $X$. 

\begin{definition}[Sobolev space $W^{1,2}_{C,\loc}(TM)$]
The Sobolev space $W^{1,2}_{C, \loc}(TM)$ is the space consisting $X \in L^2_\loc(TM)$ for which there exists a $T\in  L^2_\loc(TM) \otimes L^2_\loc(TM)$ such that 
\[
\int hT: (\nabla g_1 \otimes \nabla g_2)\,\d \mm=-\int \la X, \nabla g_2 \ra {\rm div}(h \nabla g_1)-h\H_{g_2}(X, \nabla g_1)\,\d \mm
\]  for any $g_1, g_2, h\in {\rm TestF}(M)$.
In this case we call $T$ the covariant derivative of $X$ and denote it by $\nabla X$. We endow $W^{1,2}_{C, \loc}(TM)$ with the (extended) norm $\| \cdot \|_{W^{1,2}_C(TM)}$ defined by
\[
\| X \|^2_{W^{1,2}_C(TM)}:=\| X\|^2_{L^2(TM)} +\|| \nabla X|_{\rm HS}\|^2_{L^2(M)}.
\]
We define $W^{1,2}_C(TM)$ as the Banach space consisting of   $X\in W^{1,2}_{C, \loc}(TM)$ with  finite norm.

\end{definition}
\bigskip

We recall that the set of test vector fields ${\rm TestV}(M) \subset L^2(TM)$  is defined  as
\[
{\rm TestV}(M):=\Big \{\mathop{\sum}_{i=1}^n  g_i   \nabla  f_i: n\in \mathbb{N}, f_i, g_i \in  {\rm TestF}(M), i= 1, ... , n \Big \}.
\]

When $M$ is ${\rm RCD}$, it can be proved (cf. \cite{G-N}) that ${\rm TestV}(M)$ is dense in $L^2(TM)$ and ${\rm TestV}(M) \subset W^{1,2}_C(TM)$. In particular,  for any $f \in {\rm TestF}(M)$ we have
$\nabla f \in W^{1,2}_C(TM)$ and $(\nabla \nabla f)^b=\H_f$ where $L^2(TM)\otimes L^2(TM) \ni X\mapsto X^b \in L^2(T^*M)\otimes L^2(T^*M)$ is the usual isomorphism.

We define $W^{2,2}_\loc(M)$ as the space of  functions $f\in W^{1,2}_\loc(M)$ with $\nabla f \in W^{1,2}_{C, \loc}(TM)$, equipped with the (extended) norm  
\[
\| f \|^2_{W^{2,2}(M)}:=\| |\D f|\|^2_{L^2(M)} +\|| \nabla\nabla f|_{\rm HS}\|^2_{L^2(M)}.
\]
We define $W^{2,2}(M)$ as the subspace of $W^{2,2}_\loc(M)$ consisting of vectors with finite norm.
We call $(\nabla \nabla f)^b$  the Hessian of $f$ and denote it by $\H_f$. It can be seen that this notation is compatible with \eqref{eq:hessian} when $f\in {\rm TestF}$.
We define $H^{2,2}(M) \subset W^{2,2}(M)$ as the $W^{2,2}$-
closure of ${\rm TestF}(M)$.

\begin{definition}[Sobolev space $H^{1,2}_C(TM)$]
We define the  Sobolev space $H^{1,2}_C(TM)\subset W^{1,2}_C(TM)$ as the $W^{1,2}_C(TM)$-closure of ${\rm TestV}(M)$.
\end{definition}

\bigskip

We  have the following proposition concerning  $H^{1,2}_C(TM)$ vectors, which extends the result in \cite{S-S}.

\begin{proposition}[Proposition 3.4.6, \cite{G-N}]\label{prop:close}
Let $X \in H^{1,2}_C(TM)$. Then $\la X, Y \ra \in W^{1,2}(M)$ for any  $Y \in W^{1,2}_C(TM)$.  In particular,  
\[
\nabla Y:( \nabla g \otimes \nabla h)=\la \nabla g, \nabla \la Y, \nabla h \ra \ra -\H_h(Y, \nabla g)
\]
for any $h \in {\rm TestF}(M)$. Similar property also holds for $Y \in W^{1,2}_{C, \loc}(TM)$.
 
\end{proposition}

We  define the symmetric part of  $\nabla X$ by
\[
\nabla^s X: (\nabla f \otimes \nabla g):=\frac12 \Big ( \nabla X: (\nabla f \otimes \nabla g)+\nabla X: (\nabla g \otimes \nabla f) \Big )
\]
for any $f, g \in {\rm TestF}(M)$. 
In particular, for $X \in W^{1,2}_C(TM)$ we know 
\[
\nabla^s X: (\nabla f \otimes \nabla f)=\la \nabla f, \nabla \la X, \nabla f \ra \ra -\frac12\la X, \nabla  |\D f|^2 \ra
\]
for any $f, g \in {\rm TestF}(M)$.

\bigskip

We have the following improved Bochner inequality, a more refined version for $\rcdkn$ space can be found in \cite{H-R}.

\begin{proposition}[Improved Bochner inequality, \cite{G-N}] \label{prop:bochnerimprove}
Let $M=\ms$ be  a $\rcd$ space. Then for any $f \in {\rm TestF}(M)$ we have
\[
{\bf \Gamma}_2(f) \geq \Big {(} k |\D f|^2+ |\H_f|^2_{\rm HS} \Big{)}\,\mm,
\]
where $|\H_f|_{\rm HS}$ is the Hilbert-Schmidt norm of the Hessian (as a bi-linear map). In case $M$ is $\rcdkn$,  $|\H_f|_{\rm HS}$ can be computed by local coordinate (see Proposition \ref{prop:finitedim} below).
\end{proposition}

As a corollary, we have the following important proposition.

\begin{proposition}[Corollary 3.3.9, Proposition 3.3.18, \cite{G-N}]\label{coro:hessianbound}
Let $M=\ms$ be  a $\rcd$ space. Then for any $f \in W^{1,2}(M)$ with $\Delta f \in L^2$,  we have
\[
\| |\H_f|_{\rm HS}\|^2_{L^2}\leq \| \Delta f \|^2_{L^2}-k\| |\D f|\|^2_{L^2}.
\]
Furthermore, we know $\overline{\Big \{f: f\in W^{1,2}, \Delta f \in L^2 \Big \}}^{W^{2,2}}=H^{2,2} \subset  W^{2,2}$. 
\end{proposition}

\bigskip

At the end of this part, we review some results about the  dimension of $M$,  which is understood as the dimension of  $L^2(TM)$. The definitions and basic properties on local independence, local basis and local dimension can be found in \cite{G-N} (see also \cite{H-R}).

\begin{proposition}[Theorem 1.4.11, \cite{G-N}]\label{decomposition}
Let $\ms$ be a $\rcd$ metric measure space. Then there exists  a unique Borel decomposition $\{ E_n\}_{n \in \mathbb{N} \cup \{\infty\}}$ of $X$ such that
\begin{itemize}
\item For any $n \in \mathbb{N}$ and any $B \subset E_n$ with finite positive measure,  $L^2(TM)$ has a unit orthogonal basis $\{e_{i,n}\}_{i=1}^n$ on $B$,
\item For every subset $B$ of $E_\infty$ with finite positive measure, there exists a set of  unit orthogonal vectors $\{e_{i,B}\}_{i \in \mathbb{N} \cup \{\infty\}} \subset L^2(TM)\restr{B}$  which generates $L^2(TM)\restr{B}$,
\end{itemize}
where unit orthogonal of a countable set $\{v_i\}_i\subset L^2(TM)$ on $B$ means $\la v_i, v_j \ra=\delta_{ij}$ $\mm$-a.e. on $B$.
\end{proposition}

\begin{definition}[Analytic Dimension]
We say that the dimension of $L^2(TM)$ is $k$ if $k=\sup \{n: \mm(E_n)>0\}$ where $\{ E_n\}_{n \in \mathbb{N} \cup \{\infty\}}$ is the decomposition given in Proposition \ref{decomposition}. We define the analytic dimension of $M$ as the dimension of $L^2(TM)$ and denote it by $\dim_{\rm max} M$.
\end{definition}

\bigskip

Combining Proposition 3.2 in \cite{H-R} and  Proposition \ref{coro:hessianbound},  we have the following  result about the analytic dimension of $\rcdkn$ space.

\begin{proposition}\label{prop:finitedim}
Let $M=\ms$ be a $\rcdkn$ metric measure space. Then ${\dim}_{\rm max} M \leq N$. Furthermore, if the local dimension on a Borel set $E$ is $N$, we have
$\tr \H_f(x)=\Delta f (x)$ $\mm$-a.e. $x \in E$ for every $f \in W^{1,2}(M)$ with $\Delta f \in L^2$.
\end{proposition}

\subsection{Continuity equation on metric measure spaces}
In this part we  review some recent results on continuity equation on metric measure spaces (cf. \cite{GH-C}). Here we  assume that  the metric measure space $\ms$ is $\rcd$.  Under this assumption, we know $W^{1,2}\ms$ is separable (cf. \cite{ACD-S}) so that the continuity equation can be defined pointwisely as follows.

\begin{definition}[Solutions to $\partial_t\mu_t=L_t$]\label{def:solcont}
Let $\ms$ be a metric measure space. Assume that $(\mu_t)$ is a $W_2$-continuous curve  with bounded compression (i.e. $\mu_t\leq C\mm$ for some constant $C$),  and  $ \{L_t\}_{t\in[0,1]}$ is a family of maps from $W^{1,2}\ms$ to $\R$.

We say that $(\mu_t)$ solves the continuity equation
\begin{equation}
\label{eq:basecont}
\partial_t\mu_t=L_t,
\end{equation}
provided:
\begin{itemize}
\item[i)] for a.e. $t\in [0,1]$, $W^{1,2} \ni f \mapsto L_t(f)$ is a bounded linear  functional, and $\|L_t\| \in L^2([0,1])$,

\item[ii)]  for every $f\in L^1\cap W^{1,2}$ the map $t\mapsto \int f\,\d\mu_t$ is absolutely continuous and the identity
\[
\frac{\d}{\d t}\int f\,\d\mu_t=L_t(f)
\]
holds for a.e. $t$.
\end{itemize}
\end{definition}

  We say that a curve $(\mu_t)$ in Wasserstein
space has bounded compression if $\mu_t<C\mm$ for some $C>0$.

\begin{proposition}[Continuity equation on metric measure space, \cite{GH-C}]\label{prop-conteq}
Let $\ms$ be a $\rcd$ space,  $(\mu_t)$ be a continuous curve with bounded compression in Wasserstein space. Then the following are equivalent.
\begin{itemize}
\item[i)]  $(\mu_t)$ is  2-absolutely continuous with respect to $W_2$.
\item[ii)] There is a family of maps $\{L_t\}_{t\in[0,1]}$ from $W^{1,2}$ to $\R$ such that $(\mu_t)$ solves the continuity equation according to Definition \ref{def:solcont}.
\end{itemize}

Furthermore, if  the above characterizations hold, we have
\[
\|L_t\|=|\dot\mu_t|,\qquad a.e.\ t\in[0,1].
\]

\end{proposition}

\bigskip

As an application of the Proposition \ref{prop-conteq}, we can  prove the following result concerning the derivative of $W_2^2(\cdot,\nu)$ along an absolutely continuous curve. 
\begin{proposition}[Derivative of ${W}^2_2(\cdot, \nu)$, Proposition 3.10,  \cite{GH-C}]\label{prop:derw2}
Let $\ms$  be a $\rcd$ space. Assume that $(\mu_t)\subset \mathcal{W}_2(X)$ is an absolutely continuous  curve with bounded compression,  $\nu$ has bounded support.  Then for a.e. $t\in[0,1]$  we have the formula
\begin{equation}
\label{eq:derw2}
\frac{\d}{\d t}\frac12 W_2^2(\mu_t,\nu)=L_t(\varphi_t),
\end{equation}
 where $\varphi_t$ is any Kantorovich potential from $\mu_t$ to $\nu$.
\end{proposition}

\bigskip

Next,  we discuss more about geodesics in Wasserstein space. Firstly, we recall the Hopf-Lax formula for Hamilton-Jacobi equation.

\begin{definition}
\begin{equation}
\Q_t(\phi)(x):=
\left\{
\begin{array}{ll}
\mathop{\inf}_{y \in X} c(x,y)+\phi(y)~~~~~t>0\\
\phi(x)~~~~~~~~~~~~~~~~~~~~~~~~~~t=0
\end{array}\right.
\end{equation}
where $c(x,y)=\frac{\d^2(x,y)}{2t}$, $t>0$.
\end{definition}

It is known that $ t\mapsto \Q_t (f)$ is a continuous semigroup  for any lower semi-continuous and bounded function $f$. In particular, $\mathop{\lim}_{t\rightarrow 0} \Q_t (f) =f$. Furthermore, we have the following results concerning metric Hamilton-Jacobi equation.

\begin{lemma}[Solution of Hamilton-Jacobi equation]\label{lemma-hj1}

For every $x \in X$ it holds:
\[
\frac{\d}{\d t} \Q_t (f)(x) +\frac{1}{2}|\lip {\Q_t (f)}|^2(x) =0
\]

with at most countably many exceptions in $(0,+\infty)$.
\end{lemma}

We have the following proposition about the  evolution of Kantorovich potentials by Hopf-Lax formula (see Theorem 7.36 in \cite{V-O} or Theorem 2.18 in \cite{AG-U} for a proof). 

\begin{proposition}[Evolution of Kantorovich potentials]\label{prop-1}
Let $(X, \d)$ be a metric space, $(\mu_t)_t$ be a $W_2$-geodesic in Wasserstein space and $\varphi$ be a Kantorovich potential from $\mu_0$ to $\mu_1$. Then for every $t\in [0,1]$:
\begin{itemize}
\item [1)] the function $t\Q_t(-\varphi)$ is a Kantorovich potential from $\mu_t$ to $\mu_0$,
\item [2)] the function $(1-t)\Q_{1-t}(-\varphi^c)$ is a Kantorovich potential from $\mu_t$ to $\mu_1$.
\end{itemize}

\end{proposition}

Moreover, we have the following proposition about the Wasserstein geodesic.

\begin{proposition}[Continuity equation of geodesics, \cite{GH-C}]\label{prop:geod}
Let  $(\mu_t)$  be a geodesic with bounded compression such that $\mu_0,\mu_1$ have bounded supports,  and $\varphi$ a Kantorovich potential from $\mu_0$ to $\mu_1$ which is bounded supported.
Then
\[
\partial_t\mu_t+\nabla\cdot(\nabla\phi_t \mu_t)=0,
\]
where $\phi_t:=-\Q_{1-t}(-\varphi^c)$ for every $t\in[0,1]$. 

Similarly, 
\[
\partial_t\mu_t+\nabla\cdot(\nabla\varphi_t \mu_t)=0,
\]
where $\varphi_t:=\Q_t(-\varphi)$ for every $t\in[0,1]$. 
\end{proposition}

\bigskip

At last, we recall the $C^1$-regularity of geodesics.

\begin{proposition}[Weak $C^1$-regularity for geodesics,  Proposition 5.7 \cite{GH-C} and Corollary 5.7  \cite{G-S}]\label{prop:c1}
Let $(\mu_t)\subset \mathcal{P}_2(X)$ be a geodesic with  bounded compression. Assume further  that $\mu_0, \mu_1$ have bounded supports. We denote the density of $\mu_t$ by $\rho_t$. Then for any $t \in [0, 1]$ and any sequence $(t_n)\subset [0,1]$ converging
to $t$, there exists a subsequence $(t_{n_k} )$ such that
\[
\rho_{t_{n_k}} \to \rho_t,~~~\mm-\text{a.e.}
\]
as $k\to \infty$.
Furthermore, $(\mu_t)$ is a weakly $C^1$ curve in the sense that   $t\mapsto \int f\,\d\mu_t$ is $C^1$ for any $f\in W^{1,2}$.
\end{proposition}

%%%%%%%%%%%%%%%%%%%%%%%%%%%%%%%%%%%%%%%%%
\section{Main results}

\subsection{Regular Lagrangian flow}
In  this part  we will review the existence and uniqueness theory of continuity equation,  and {\sl Regular Lagrangian Flow} (RLF for short) on metric measure spaces studied by Ambrosio-Trevisan in \cite{AT-W}. Then we will prove some basic results which will be used in the proof of our main theorems. 

\begin{definition}[Regular Lagrangian flow]
We say that a measurable map  $F: X \times  [0, T] \mapsto X$ is a {\sl regular Lagrangian flow} associated to ${\bf b}\in L^2_\loc(TM)$ if :
\begin{itemize}
\item[1)]  $F_0(x)=x$  and $(F_t(x))_t \in {\rm C}([0,T], X)$ for all $x\in X$. 

\item[2)]For any $f\in W^{1,2}\ms$, we have $f\circ F_t(x) \in W^{1,1}([0,T])$ and
\[
\frac{\d}{\d t} f\circ F_t(x)=\la {\bf b}, \nabla f\ra \circ F_t(x)
\]
for $L^1 \times \mm$-a.e. $(t, x)$.

\item[3)] There exists a constant $C_0(T)>0$ such that $(F_t)_\sharp (\mm) \leq C_0\mm$ for all $t\in [0,T]$.

\end{itemize}
\end{definition}

\begin{proposition}[Existence and uniqueness theory, Ambrosio-Trevisan \cite{AT-W}]\label{prop:rlf}
Let ${\bf b}\in L^2_\loc(TM)$ be a vector field with $|{\bf b}| \in L^2+L^\infty$,  ${\bf b}\in W^{1,2}_{C, \loc}(TM)$, $| \nabla {\bf b}|_{\rm HS} \in L^2$,  ${\rm div} {\bf b}\in L^2+L^\infty$ and $\big ({\rm div} {\bf b}\big )^- \in L^\infty$.
There exists a unique regular Lagrangian flow $F: X \times  [0, T] \mapsto X$  such that
\begin{itemize}
\item[1)]For any initial condition $\mu_0=f\mm$ with $f\in L^1 \cap L^\infty$,  $\mu_t:=(F_t)_\sharp \mu_0$ is a solution to the continuity equation
\[
\frac{\d}{\d t} \int g\,\d \mu_t=\int \la \nabla g, {\bf b} \ra\,\d \mu_t, ~~L^1-\text{a.e.}~t\in (0,T),~~\lmt{t}{0}\int g\,\d \mu_t=\int g\,\d \mu_0
\]
for any $g \in \Lip(X ,\d)  \cap L^\infty$. Moreover,  $\frac{\d \mu_t}{\d \mm} \in L^1 \cap L^\infty$ for any $t$.
\item[2)] For $\mm$-a.e. $x$, $|\dot{F_t}|(x)=|{\bf b}|(F_t(x))$ a.e. $t\in (0, T)$.

\item[3)] Let  $\mu_0=f\mm$ be a probability measure with $f\in L^2$, $\mu_t=(F_t)_\sharp \mu_0$. Then
\[
\left \| \frac{\d \mu_t }{\d \mm} \right \|_{L^2} \leq e^{C_1 t}\| f\|_{L^2} 
\]
for some constant $C_1$ which depends  on $ \|({\rm div} {\bf b} )^-\|_{L^\infty}$.

\item[4)]$F_t$ is unique(or non-branching) in the following sense. If $\bar{F}_t$ is another map satisfying the properties above, then  $(\bar{F}_t)_\sharp \mu=(F_t)_\sharp \mu$ for any $\mu\in \mathcal{P}(X)$ with bounded density. As a consequence, we know that $(F_t)$ is a semigroup in the sense that $F_{t+s}(x)=F_t\circ F_s (x)$, $\mm$-a.e. $x\in X$ for any $s, t, s+t \in [0, T]$.
\end{itemize}

\end{proposition}

\bigskip

In some potential applications, we do not have the global $L^2+L^\infty$-bound for $|{\bf b}|$, ${\rm div} {\bf b}$ or  global $L^\infty$-bound for $\big ({\rm div} {\bf b}\big )^-$. So we need the following proposition (c.f. Theorem 4.2 \cite{GKKO-R}).

\begin{proposition}\label{prop:rlf-2}
Let ${\bf b}\in W^{1,2}_{C, \loc}(TM)$.  Assume that  $|{\bf b}| \leq C_0 \d(x, x_0)+C_1, \mm$-a.e. for some $C_0, C_1>0, x_0\in X$,  and $| \nabla {\bf b}|_{\rm HS}\in L^2(\Omega) $, $ {\rm div} {\bf b}\in L^2(\Omega)+L^\infty (\Omega)$, $\big({\rm div} {\bf b}\big)^-\in L^\infty (\Omega)$ for any bounded set $\Omega$. Then there exists a unique regular Lagrangian flow associated to the vector field ${\bf b}$.
\end{proposition}

\begin{proof}
Let $\mu \in \mathcal{P}_2(X)$ be an arbitrary measure with bounded density. We assume that $\supp \mu \in B_R(x_0)$ for some $R\geq 1$. Let $\chi$ be a cut-off function in Lemma 
6.7 \cite{AMS-O} such that $\chi$ is Lipschitz and 
\begin{itemize}
\item [a)] $0\leq \chi \leq 1$, $\chi$ supports on $B_{3R}(x_0)$ and $\chi=1$ on $B_{2R}(x_0)$,
\item [b)] $\Delta \chi \in L^\infty$ and $|\D \chi|^2 \in W^{1,2}$.
\end{itemize}
Then we know $|\chi {\bf b}| \in L^2+L^\infty$,
$\nabla (\chi {\bf b})=\chi \nabla {\bf b}+\nabla \chi \otimes {\bf b} \in L^2(TM)\otimes L^2(TM)$, and  $\chi {\bf b} \in {\rm D}({\rm div})$, ${\rm div}(\chi {\bf b})=\la {\bf b},\nabla \chi \ra +\chi {\rm div} {\bf b} \in L^2+L^\infty$,   so that
\[
\| \big ({\rm div}(\chi {\bf b})\big)^-\|_{L^\infty} \leq \| |{\bf b}||\nabla \chi|\|_{L^\infty}+\|\chi \big( {\rm div}{\bf b}\big)^- \|_{L^\infty}<\infty.
\] 

From Proposition \ref{prop:rlf}, we know the regular Lagrangian flow associated to $\chi {\bf b}$ exists and we denote this flow by $\bar{F}_t$. The curve $\mu_t:=(\bar{F}_t)_\sharp \mu$ is the unique solution to the continuity equation
\[
\frac{\d}{\d t} \mu_t+{\rm div}(\chi {\bf b}\mu_t) =0,~~~~~\mu_0=\mu.
\]
In particular, when $\supp \mu_t \subset B_{2R}$, we know $|{\bf b}|(x) \leq 2C_0R +C_1$,  for $\mm$-a.e. $x\in \supp \mu_t$. From $4)$ of Proposition \ref{prop:rlf},  we know $\supp \mu_t \subset B_{2R}$ when $t\in [0, \frac{2R-R}{2C_0R+{C_1}}]$. So
\begin{equation}\label{eq1-prop-rlf}
\frac{\d}{\d t} \mu_t+{\rm div}({\bf b}\mu_t) =0,~~~~~t\in[0, \frac1{2C_0+{C_1}}],~~ \mu_0=\mu.
\end{equation}
Then  for any $T>0$, we can find a solution to the continuity equation \eqref{eq1-prop-rlf} for $t\in [0, T]$ by repeating the  construction above for finite times. It can be seen from the 
construction that 
this solution is unique.

Finally, we can prove the existence  and uniqueness of regular Lagrangian flow using Theorem 8.3 \cite{AT-W} and the proof therein.
\end{proof}

\bigskip

For convenience,  we will not distinguish  the regular Lagrangian flow $(F_t)$  and the curve in Wasserstein space push-forward by $F_t$. We will see in 
Proposition \ref{prop:c1-2} that the curve push-forward by $F_t$ is $C^1$. To prove this result, we firstly recall a useful lemma.

\begin{lemma}[``Weak-strong" convergence, Lemma 5.11 \cite{G-S}]\label{lemma:weakstrong}
Let $\ms$ be an infinitesimally Hilbertian space.  Assume that
\begin{itemize}
\item[i)]   $(\mu_n)\subset\mathcal{P}(X)$  is a sequence of measures with uniformly bounded densities, such that  $\rho_n\to \rho$ $\mm$-a.e. for some probability density  $\rho$, where  $\mu_n:=\rho_n \mm$ and $\mu:=\rho \mm$,
\item[ii)]  $(f_n) \subset W^{1,2}$ is a sequence such that
\[
{\sup}_{n\in \mathbb{N}} \int |\D f_n|^2\,\d\mm <\infty,
\]
and  $f_n \to f$ $\mm$-a.e. for some Borel function $f$.

\end{itemize}
Then for any ${\bf b}\in L^2(TM)$,  we have
\[
\mathop{\lim}_{n\rightarrow \infty}\int \la \nabla f_n, {\bf b}\ra \,\d\mu_n =\int \la \nabla f,  {\bf b}\ra \,\d\mu.
\]

\end{lemma}
\begin{proof}
If ${\bf b}=\nabla g$ for some $g\in W^{1,2}$,  the assertion has been proved in Lemma 5.11 \cite{G-S}. For any $\epsilon >0$, we can find $v_\epsilon \in {\rm TestV}$ with $v_\epsilon=\sum_i^N a_i \nabla g_i$ such that $\|{\bf b}-v_\epsilon\|_{L^2(TM)} <\epsilon$. Then we have
\begin{eqnarray*}
\mathop{\lims}_{n\rightarrow \infty}\int \la \nabla f_n, {\bf b}\ra \,\d\mu_n &\leq & \mathop{\lims}_{n\rightarrow \infty}\int \la \nabla f_n, {\bf b}-v_\epsilon \ra\,\d\mu_n+ \mathop{\lims}_{n\rightarrow \infty}\int \la\nabla f_n, v_\epsilon \ra\,\d\mu_n\\
&\leq&  \mathop{\lims}_{n\rightarrow \infty}\sum_i^N\int \la \nabla f_n, \nabla g_i\ra \,a_i\d\mu_n+{\rm O}(\epsilon)\\
&=&  \sum_i^N\int \la \nabla f, \nabla g_i\ra  \,a_i\d\mu+{\rm O}(\epsilon)\\
&\leq& \int \la \nabla f,  {\bf b}\ra \,\d\mu+{\rm O}(\epsilon).
\end{eqnarray*}
Letting $\epsilon \to 0$ and considering the opposite inequality we prove the assertion.
\end{proof}

\bigskip

\begin{proposition}\label{prop:c1-2}
Let $(F_t)$ be a regular Lagrangian flow  associated to ${\bf b}\in W^{1,2}_{C, \loc}$. Assume that $\mu_0$ has bounded density and bounded support. Then $\mu_t:=(F_t)_\sharp \mu_0$   is a $C^1$ curve (see also Proposition \ref{prop:c1}). 
\end{proposition}
\begin{proof}
Let $ \mu_t:=(F_t)_\sharp \mu_0$ be a RLF with $\rho_t:=\frac{\d \mu_t }{\d \mm}$ uniformly bounded in $t$. By $3)$ of Proposition \ref{prop:rlf} we have 
$$\lmts{t}{0} \| \rho_t\|_{L^2} \leq \|\rho_0\|_{L^2}.
$$
 It is known that the functional $\mathcal{P}_2(X) \ni \mu \mapsto  \int \big (\frac{\d \mu }{\d \mm}  \big)^2\,\d \mm$ is lower semi-continuous in Wasserstein space (cf. \cite{AGS-G}). So the function $t \mapsto \| \rho_t\|_2$ is lower semi-continuous. Then we have $\lmt{t}{0} \| \rho_t\|_{L^2} = \|\rho_0\|_{L^2}$.

Since $\rho_t \to \rho_0$ weakly in duality with $C_b(X)$ and $(\rho_t) $ is uniformly bounded in $L^2$. We know  that $\rho_t \to \rho_0$ weakly in $L^2(X, \mm)$.  Combining with $\lmt{t}{0} \| \rho_t\|_{L^2} = \|\rho_0\|_{L^2}$ we know $\rho_t \to \rho_0$ in $L^2$ strongly,  and in $L^p$ strongly for any $p\in [1, \infty)$.

From semi-group property,  we know $t \mapsto \rho_t$ is  continuous in $L^1$. So for any $t, (t_n)_n \geq 0$ with $t_n \to t$, we know there exists a subsequence $(t_{n_k})_k$ such that $\rho_{t_{n_k}} \to \rho_t$ $\mm$-a.e.  as $k \to \infty$.  By Lemma \ref{lemma:weakstrong} we can prove the continuity of  the function
\[
[0, T] \ni r\mapsto \int \la \nabla f,  {\bf b}\ra\,\d \mu_{r}.
\] So  $(\mu_t)$ is a $C^1$ curve.
\end{proof}

\bigskip

The following simple lemma is a complement to Proposition \ref{prop:rlf}.
 
\begin{lemma}\label{lemma-rlf-1}
Let $f\in W^{1,2}$, ${\bf b} \in L^2_\loc(TM)$. We assume that $(F_t)_t$ is a regular Lagrangian flow associated to ${\bf b}$. If $f\circ F_t \in W^{1,2}$  for any $t>0$. Then for all $t\in [0,T]$,
\[
\la {\bf b}, \nabla f \ra\circ F_t(x)= \la {\bf b}, \nabla (f\circ F_t)\ra (x),~~\mm-\text{a.e.}~x\in X.
\]
\end{lemma}
\begin{proof}
Let $\mu_0 \in \mathcal{P}(X)$ be an arbitrary measure with bounded density and bounded support. We define  $\mu_t=(F_t)_\sharp \mu_0, t>0$. From the definition of continuity equation and Proposition \ref{prop:rlf},  we know 
\[
\frac{\d }{\d t} \int f \,\d \mu_t=\int \la {\bf b}, \nabla f \ra \,\d \mu_t=\int \la {\bf b}, \nabla f \ra\circ F_t \,\d \mu_0
\]
for a.e. $t\in [0,T]$. From Proposition \ref{prop:c1-2} above we know this formula holds for all $t$.
Meanwhile, since  $f\circ F_{t+h} \in W^{1,2}$ for any $h>0$, we know
\begin{eqnarray*}
\frac{\d }{\d h} \int f \,\d \mu_{t+h}\restr{h=0}&=&\frac{\d }{\d h} \int f\circ F_t \,\d \mu_{h}\restr{h=0}\\
&=& \int \la {\bf b}, \nabla (f\circ F_t) \ra \,\d \mu_0.
\end{eqnarray*}

Then we have
\[
\int \la {\bf b}, \nabla (f\circ F_t) \ra \,\d \mu_0=\int \la {\bf b}, \nabla f \ra\circ F_t \,\d \mu_0.
\]
As $\mu_0$ is arbitrary, we know $\la {\bf b}, \nabla f \ra\circ F_t= \la {\bf b}, \nabla (f\circ F_t)\ra $, $\mm$-a.e..
\end{proof}

\subsection{$K$-convexity and $K$-monotonicity}
First of all, we introduce some notions and concepts to characterize the convexity of functions, and the monotonicity of vector fields in non-smooth setting.

The first one is a zero order characterization. 

\begin{definition}[Weak $K$-convexity]\label{def:convex}
Let $u\in L^1_{\rm loc}(X, \mm)$. We define the functional $U(\cdot): \mathcal{P}_2(X) \ni \mu \mapsto \R\cup \{+\infty\}$  by
\begin{equation*}
U(\mu):=
\left \{\begin{array}{ll}
 \int_X u\,\d \mu  ~~~~~~~~~~~\text{if}~~\mu\ll \mm,\\
+\infty  ~~~~~~~~~~~~~~~\text{otherwise}.
\end{array}\right.
\end{equation*}

 We say that $u$ is {\sl weakly $K$-convex} if the functional $U(\cdot)$ is $K$-convex on Wasserstein space:
\begin{equation}\label{def1-eq}
U(\mu_t) \leq (1-t)U(\mu_0)+tU(\mu_1)-\frac{K}{2}(1-t)tW^2_2(\mu_0, \mu_1)
\end{equation}
for any $t\in [0,1]$ along any geodesic $(\mu_t) \subset (\mathcal{P}_2, W_2)$,  where $\mu_0, \mu_1$ have bounded densities and bounded supports.
\end{definition}

\bigskip

The second one is a first order characterization.

\begin{definition}[$K$-monotonicity]\label{def:mono}
We say that a vector field ${\bf b}\in L^2_\loc(TM)$ is {\sl $K$-monotone} if
\[
\int \la {\bf b}, \nabla \varphi \ra\,\d \mu^1+\int \la {\bf b}, \nabla (\varphi)^c \ra\,\d \mu^2 \geq K W^2_2(\mu^1,\mu^2)
\]
for any $\mu^1, \mu^2 \in \mathcal{P}_2(X)$ with bounded densities and bounded supports, where $(\varphi, \varphi^c)$ is a couple of Kantorovich potentials relative to $(\mu^1, \mu^2)$.
\end{definition}

\begin{remark}
From the locality  property of Kantorovich potentials in the following Proposition \ref{prop:brenier}, we know that the $K$-monotonicity is independent of the choice of $(\varphi, \varphi^c)$, so this concept is well-defined.

If ${\bf b}\in L^2(TM)$, by the locality of Kantorovich potentials again, we can replace the  condition ``$\mu^1, \mu^2 \in \mathcal{P}_2$ with bounded supports and bounded densities" in Definition \ref{def:mono} by ``bounded densities". 

Similarly, by metric Brenier's theorem we can rephrase Definition \ref{def:convex} in the following way: for any $\mu_0, \mu_1\in \mathcal{P}_2(X)$ with bounded densities and bounded supports, there exists a geodesic $(\mu_t)_t \subset (\mathcal{P}_2, W_2)$ connecting $\mu_0, \mu_1$ such that the inequality \eqref{def1-eq} holds.
\end{remark}

\begin{proposition}[Metric Brenier's theorem, \cite{AGS-M, RS-N}]\label{prop:brenier}
Let $\ms$ be a $\rcd$ metric measure  space. Assume that  $\mu,\nu\in \mathcal{P}_2$ are absolutely continuous with respect to $\mm$. Let $\varphi$ be a Kantorovich potential relative to $(\mu, \nu)$. Then the geodesic   connecting $\mu$ and $\nu$ is unique.  The lifting  $\Pi$  of this geodesic   $(\mu_t)$  is  induced by a map and $\Pi$ concentrates on a  set of non-branching geodesics. Moreover, for $\Pi$-a.e. $\gamma \in \geo(X)$ we have 
\[
\d(\gamma_0,\gamma_1)=\lip{\varphi}(\gamma_0)=|\D\varphi|(\gamma_0).
\]
In particular,  we have
\[
W_2(\mu, \nu)= \sqrt{\int |\D\varphi|^2  \, \d\mu}.
\]

Furthermore,  we have the locality property of Kantorovich potentials: 
\[
|\D(\varphi-\bar{\varphi})|=0~~~~\mm-\text{a.e. on}~\supp \mu
\]
for any  $\varphi, \bar{\varphi}$ which are  Kantorovich potentials from $\mu$ to $\nu$.
\end{proposition}

\bigskip

Next, we introduce the concept of {\sl infinitesimal $K$-monotonicity} of a vector field ${\bf b}\in W^{1,2}_{C,\loc}(TM)$, which is a second order characterization.  We recall  that the Hessian of a test function $f$ can be defined by
\[
2\H_f(\nabla g_1, \nabla g_2)=\la \nabla \la \nabla f, \nabla g_1\ra, \nabla g_2\ra+\la \nabla \la \nabla f, \nabla g_2 \ra, \nabla g_1 \ra-\la \nabla \la \nabla g_1, \nabla g_2\ra, \nabla f  \ra,
\]
and the covariant derivative of a vector field ${\bf b}\in W^{1,2}_{C,\loc}(TM)$ can be equivalently defined by
\[
\nabla {\bf b}:(\nabla g_1 \otimes \nabla g_2)=\la \nabla \la {\bf b}, \nabla g_2\ra, \nabla g_1\ra-\H_{g_2}(\nabla g_1, {\bf b}),
\]
where $g_1, g_2 \in {\rm TestF}$.

\begin{definition}[Infinitesimal $K$-monotonicity]\label{def:infmono}
Let ${\bf b} \in W^{1,2}_{C, \loc}(TM)$ be a vector field. We say that ${\bf b}$ is {\sl infinitesimally $K$-monotone} if
\[
\nabla^s {\bf b}:(X \otimes X)=\nabla {\bf b}: (X \otimes X)  \geq K|X|^2~~~~\mm-\text{a.e.}
\]
for any $X \in L^2(TM)$. 
\end{definition}

\bigskip
\begin{definition}[Infinitesimal $K$-convexity]
We say that $f$ is {\sl infinitesimally $K$-convex} if   $\nabla f\in W^{1,2}_{C, \loc}(TM)$ and $\nabla f$ is infinitesimally $K$-monotone. In other words, $f$ is infinitesimally $K$-convex if $f\in W^{2,2}_\loc$ and  $\H_f(\nabla g, \nabla g) \geq K|\D g|^2$ for any $g\in {\rm TestF}$.
\end{definition}

\bigskip

Next we  prove the first theorem in this article.  When $u\in {\rm TestF}$, this  result  has been proved  in Theorem 7.1 \cite{K-O}  (see also Lemma 2.1  \cite{LierlSturm2018},  Theorem 3.3  \cite{GKKO-R}).  In the following Theorem \ref{mainth-0}, thanks to the recent results on second order differential structure of metric measure space (cf. \cite{G-N}), we can remove some bounds on $u, \nabla u$,  and the condition $\Delta u\in W^{1,2}$ in the former proofs.

\begin{theorem}\label{mainth-0}
Let $M:=\ms$ be a $\rcd$ metric measure space, $u\in W^{2,2}_\loc\ms$. Assume  further that $u\in L^\infty_\loc(M)$ and $u(x) \geq -a-b \d^2(x, x_0)$ for some $a, b \in \R$, $x_0 \in X$.
Then the following are equivalent:
\begin{itemize}
\item [i)] $u$ is infinitesimally $K$-convex.
\item [ii)] $u$ is weakly $K$-convex.
\end{itemize}

\end{theorem}

\begin{proof}

First of all, we rewrite the Bochner's formula in Proposition \ref{becondition} in the following weak form. Recall that ${\rm D}_{L^\infty}(\Delta):=\Big 
\{ \varphi: \Delta \varphi \in L^2 \cap L^\infty, \varphi\in W^{1,2}\cap L^\infty \Big \}$.

For any $f \in {\rm TestF}(M)$, $\varphi \in {\rm D}_{L^\infty}(\Delta)$, we define
\begin{eqnarray}\label{th0-eq-0}
 \Gamma_2(f; \varphi)&:=& \int \varphi \,\d {\bf \Gamma}_2(f)\\
 &=&\frac12 \int |\D f|^2 \Delta \varphi\,\d\mm-\int \la \nabla f, \nabla \Delta f\ra \varphi\,\d\mm.
\end{eqnarray}

If $\varphi\in  \Lip$, we know  $\varphi\nabla f\in {\rm D}({\rm div})$, hence
\begin{eqnarray*}
 \Gamma_2(f; \varphi)&=&\frac12 \int |\D f|^2 \Delta \varphi\,\d\mm+\int  {\rm div}(\varphi\nabla f)\Delta f \,\d\mm\\
 &=& \frac12 \int |\D f|^2 \Delta \varphi\,\d \mm+\int (\Delta f)^2\varphi\,\d \mm+\int \la \nabla \varphi, \nabla f\ra \Delta f\,\d \mm\\
 &=:&\tilde{\Gamma}_2(f; \varphi).
\end{eqnarray*}

By Proposition \ref{becondition}  we know
\begin{equation}\label{th0-eq-1}
\tilde{\Gamma}_2(f; \varphi)= \Gamma_2(f; \varphi) \geq k\int |\D f|^2\varphi\,\d\mm
\end{equation}
for any $f \in {\rm TestF}(M)$, $\varphi \in {\rm D}_{L^\infty}(\Delta)\cap \Lip $, $\varphi\geq 0$.
%添加一个证明

 Since $u$ is locally bounded, we know $W^{1,2}_\loc(M)=W^{1,2}_\loc(M^u)$ as sets. We define
$$
{\rm D}_1:=\Big\{ f\in {\rm D}({\bf \Delta})\cap L^\infty_\loc\cap \Lip(M): \Delta f\in L^2_\loc(M)\Big\} $$
and
$$
{\rm D}_2:=\Big\{ f\in {\rm D}({\bf \Delta}^{M^u})\cap L^\infty_\loc\cap \Lip(M^u):\Delta^{{M^u}} f \in L^2_\loc(M^u)\Big\},
$$
It can be seen that $D_1=D_2$ as sets and  ${\bf \Delta}^{M^u}f={\bf \Delta}f- \la \nabla u, \nabla f \ra\,\mm$ for any $f \in {\rm D}_1$.
In fact, for any $f\in {\rm D}_1$  and  $\varphi\in \Lip(X, \d)$ with bounded support, we have
\begin{eqnarray*}
\int \la \nabla f, \nabla \varphi \ra e^{-u}\,\d\mm&=&\int \la \nabla f, e^{-u}\nabla \varphi \ra \,\d\mm\\
(\text{by Leibniz rule}) &=& \int \la \nabla f, \nabla (e^{-u}\varphi )\ra \,\d\mm-\int \la \nabla f, \nabla e^{-u} \ra \varphi \,\d\mm\\
&=& -\int e^{-u}\varphi \Delta f\,\d\mm+\int \la \nabla f, \nabla u \ra \varphi e^{-u} \,\d\mm.
\end{eqnarray*}
So $f\in  {\rm D}({\bf \Delta}^{M^u})$ and ${ \Delta}^{M^u}f={ \Delta}f- \la \nabla u, \nabla f \ra \in L^2_\loc$. Similarly, we can prove the  opposite assertion.

We define ${\rm TestF}(M^u)$ as the space of test functions on $M^u:=(X, \d, e^{-u}\mm)$.  For any $f \in {\rm TestF}(M^u)$, $\varphi \in {\rm D}_{L^\infty}(\Delta^{M^u})\cap W^{1,2}\cap \Lip (M^u)$, we define $\Gamma^u_2(f; \varphi)$ as in \eqref{th0-eq-0} by replacing $(\Delta, \mm)$ by $(\Delta^{M^u}, e^{-u}\mm)$. Similarly we can define $\tilde{\Gamma}^u_2(f; \varphi)$ for any $f\in {\rm D}_2$.
From Lemma \ref{lemma:mainth-1} below we know  the following assertions are equivalent:
\begin{itemize}
\item [a)] $\Gamma_2(f; \varphi) \geq k\int |\D f|^2\varphi\,\d\mm$, and $ \int \H_u(\nabla f, \nabla f)\varphi\,\d\mm\geq K\int |\D f|^2\varphi\,\d\mm$  for any $f  \in {\rm TestF}(M)$,  $\varphi \in {\rm D}_{L^\infty}(\Delta) $, $\varphi\geq 0$,
\item [b)] $\Gamma_2^{mu}(f; \varphi) \geq (mK+k)\int |\D f|^2\varphi\,e^{-mu}\d\mm$ for any $m\in \N$, $f\in {\rm TestF}(M^{mu})$,  $\varphi \in {\rm D}_{L^\infty}(\Delta^{M^{mu}}) $, $\varphi\geq 0$.
\end{itemize}

\bigskip
Now we can complete the proof.

\underline{$i) \Longrightarrow ii)$.}

If  $u$ is infinitesimally $K$-convex.  From Lemma \ref{lemma:mainth-1}, we know a) $\Longrightarrow$ b). As $M$ has Sobolev-to-Lipschitz property, so $M^{mu}:=(X, \d, e^{-mu}\mm)$  also has such property. Since $u(x) \geq -a-b \d^2(x, x_0)$, we know $e^{-mu}\mm$  has exponential volume growth.  From Proposition \ref{becondition}, we konw $M^{mu}$ is a ${\rm RCD}(k+mK, \infty)$  space. Therefore (by the original definition of $\cd$ condition,  cf. \cite{S-O1}) we have
\begin{equation}\label{th0-eq-2}
{\rm Ent}_{ e^{-mu}\mm}(\mu_t) \leq (1-t){\rm Ent}_{ e^{-mu}\mm}(\mu_0)+t{\rm Ent}_{ e^{-mu}\mm}(\mu_1)-\frac{mK+k}2t(1-t)W^2_2(\mu_0, \mu_1)
\end{equation}
for any geodesic $(\mu_t)$ in Wasserstein space with bounded compression. Dividing $m$ on both sides of \eqref{th0-eq-2} and letting $m\to \infty$,
combining with the fact  ${\rm Ent}_{ e^{-mu}\mm}(\mu_t)={\rm Ent}_{\mm}(\mu_t)+m\int u\,\d\mu_t$, we know $u$ is weakly $K$-convex.

\underline{$ii) \Longrightarrow i)$.}

 If  $u$ is weakly $K$-convex, we know (from the  definition) that the metric measure space $M^{mu}:=(X, \d, e^{-mu}\mm)$ is  ${\rm RCD}(k+mK, \infty)$ for any $m\in N$. By Proposition \ref{becondition}  and Lemma \ref{lemma:mainth-1} we have  a).  By the density of test functions we can prove $\H_u \geq K$.

\end{proof}

\bigskip

\begin{lemma}\label{lemma:mainth-1}
The following assertions are equivalent:
\begin{itemize}
\item [1)] $\Gamma_2(f; \varphi) \geq k\int |\D f|^2\varphi\,\d\mm$, and $ \int \H_u(\nabla f, \nabla f)\varphi\,\d\mm\geq K\int |\D f|^2\varphi\,\d\mm$  for any $f  \in {\rm TestF}(M)$,  $\varphi \in {\rm D}_{L^\infty}(\Delta)$, $\varphi\geq 0$,
\item [2)] $\Gamma_2(f; \varphi) \geq k\int |\D f|^2\varphi\,\d\mm$, and $ \int \H_u(\nabla f, \nabla f)\varphi\,\d\mm\geq K\int |\D f|^2\varphi\,\d\mm$  for any $f  \in {\rm TestF}_{\rm bs}(M)$,  $\varphi \in {\rm D}_{L^\infty}(\Delta)\cap \Lip (M) $, $\varphi\geq 0$ with bounded support,
\item [3)] ${\tilde \Gamma}_2^{mu}(f; \varphi) \geq (mK+k)\int |\D f|^2\varphi\,e^{-mu}\d\mm$ for any $m\in \N$, $f \in {\rm TestF}_{\rm bs}(M)$,  $\varphi \in {\rm D}_{L^\infty}(\Delta)\cap \Lip (M)$, $\varphi\geq 0$ with bounded support,
\item [4)]  $\Gamma_2^{mu}(f; \varphi) \geq (mK+k)\int |\D f|^2\varphi\,e^{-mu}\d\mm$ for any $m\in \N$, $f\in {\rm TestF}(M^{mu})$,  $\varphi \in {\rm D}_{L^\infty}(\Delta^{M^{mu}}) $, $\varphi\geq 0$.
\end{itemize}

\end{lemma}

\begin{proof}

$1) \Longleftrightarrow 2)$ is a direct consequence of the density of ${\rm TestF}_{\rm bs}$ in ${\rm TestF}$ (see Lemma \ref{lemma:boundedsupport}). 

To prove $2) \Longrightarrow 3)$, it is sufficient to prove 
$$
{\tilde \Gamma}_2^{mu}(f; \varphi)=\Gamma_2(f; e^{-mu} \varphi)+m\int \H_{u}(\nabla f, \nabla f)\varphi\,e^{-mu}\d\mm
$$
for any $f  \in {\rm TestF}_{\rm bs}(M)$,  $\varphi \in {\rm D}_{L^\infty}(\Delta)\cap \Lip (M) $, $\varphi\geq 0$.
By Proposition \ref{prop:close} we know $\la \nabla u, \nabla f \ra \in W^{1,2}$, and  the Hessian of $u\in W^{2,2}$ can be written in the form of formula \eqref{eq:hessian}. So by  a direct computation we have
\begin{eqnarray*}
{\tilde \Gamma}_2^{mu}(f; \varphi)&=& \frac12 \int |\D f|^2  \Delta^{M^u} \varphi  \,e^{-mu} \d\mm-\int \la \nabla f, \nabla \Delta^{M^u} f\ra \varphi  \,e^{-mu}\d\mm\\
&=&  \frac12 \int |\D f|^2 ( \Delta-m\nabla u)   \varphi \,e^{-mu} \d\mm-\int \la \nabla f, \nabla (\Delta f-m\la \nabla u, \nabla f\ra) \ra \varphi \,e^{-mu}\d\mm\\
&=& \Gamma_2(f; e^{-mu} \varphi)+m\int \H_{u}(\nabla f, \nabla f)\varphi\,e^{-mu}\d\mm.
\end{eqnarray*}

 Conversely, we claim that for any $\varphi  \in {\rm D}_{L^\infty}(\Delta)\cap \Lip (M)$, $\varphi\geq 0$ with bounded support, we can find $\varphi_n \in {\rm D}_{L^\infty}(\Delta)\cap \Lip (M)$, $\varphi_n\geq 0$ with bounded support such that $\varphi_n e^{-mu} \to \varphi$ in $W^{1,2}$. To prove this claim, we recall the following well-known approximation procedure. For any $f\in L^2$, we define
\[
{\mathrm h}_\epsilon f:=\frac 1{\epsilon} \int_0^\infty\kappa({r}/{\epsilon}){\mathcal{H}}_r f\,\d r=\int^\infty_0 \kappa(s)\mathcal{H}_{\epsilon s} f\,\d s.~~~\epsilon>0,
\]
where $(\mathcal{H}_t)$ is the heat flow,  $\kappa \in C^\infty_c((0, \infty))$ with $\kappa \geq 0$ and $\int_0^\infty \kappa(r)\,\d r=1$. It can be checked that $\Delta{\mathrm h}_\epsilon f\in L^2 \cap L^\infty$, ${\mathrm h}_\epsilon f \in {\rm TestF}$ if $f\in L^2 \cap L^\infty$. In addition, we know ${\mathrm h}_\epsilon f \to f$ both in $W^{1,2}$ and ${\rm D}({\bf \Delta})$ as $\epsilon \downarrow 0$ if $f\in {\rm D}({\bf \Delta})$.

Now we turn back to our problem. Since $u, e^{-u}$ are locally finite, we can approximate $\eta e^{mu}$ by test functions $(\phi_n)_n$, where $\eta\in {\rm TestF}$ has bounded support and $\eta=1$ on $\supp \varphi$. Then $\varphi_n:=\varphi \phi_n$ achieve our aim.
From $3)$  we know
\[
\Gamma_2(f; e^{-mu} \varphi_n)+m\int \H_{u}(\nabla f, \nabla f)\varphi_n\,e^{-mu}\d\mm\geq (mK+k)\int |\D f|^2\varphi_n\,e^{-mu}\d\mm.
\]
Letting $n\to \infty$, we have
\begin{equation}\label{th0-eq-1.5}
\Gamma_2(f; \varphi)+m\int \H_{u}(\nabla f, \nabla f)\varphi\,\d\mm\geq (mK+k)\int |\D f|^2\varphi\,\d\mm.
\end{equation}
Letting $m=0$, we know $\Gamma_2(f; \varphi) \geq k\int |\D f|^2\varphi\,\d\mm$. 
Dividing $m$ on both sides of \eqref{th0-eq-1.5} and letting $m \to \infty$, we prove $\H_u \geq K$.

To prove $3) \Longrightarrow 4)$ it is sufficient to approximate $f, \varphi$ in $4)$. We firstly assume that $f\in L^\infty \cap \Lip$ and $\varphi \in \Lip$, then we can use the approximation technique above again.
Let $(\chi_n)$ be the cut-off functions in Lemma \ref{lemma:boundedsupport}. For any $n\in \N$, we can find $a_n>n$ such that 
$$
\|\chi_{a_n} f-f\|_{W^{1,2}(M^{mu})}+\|\Delta^{M^{mu}} (\chi_{a_n} f-f)\|_{L^2(M^{mu})}<\frac1n.
$$
Since $\chi_{a_n} f \in L^2 \cap L^\infty (M)$,  from the above mentioned approximation procedure, we know that 
${\mathrm h}_\epsilon (\chi_{a_n} f) \to \chi_{a_n} f$ both in $W^{1,2}(M)$ and in ${\rm D}({\bf \Delta})$ as $\epsilon \downarrow 0$. In particular, we know $\chi_{a_n}{\mathrm h}_\epsilon (\chi_{a_n} f) \to \chi_{a_n}^2 f$  in $W^{1,2}(M)$ and in ${\rm D}({\bf \Delta})$ as $\epsilon \downarrow 0$. As each of $\chi_{a_n}{\mathrm h}_\epsilon (\chi_{a_n} f)$ and $\chi_{a_n} f$ has bounded support, we know that there exits $0<b_n< \frac1n$ such that 
$$
\|\chi_{a_n} f {\mathrm h}_{b_n}(\chi_{a_n} f)-\chi^2_{a_n} f\|_{W^{1,2}(M^{mu})}+\|\Delta \big (\chi_{a_n}{\mathrm h}_{b_n}(\chi_{a_n} f)\big)-\Delta(\chi^2_{a_n} f)\|_{L^2(M^{mu})}<\frac1n.
$$

We define $f_n:=\chi_{a_n} {\mathrm h}_{b_n} (\chi_{a_n} f)$. It can be seen that $f_n \in {\rm TestF}_{\rm bs}(M)$ and 
$f_n \to f$  both in $W^{1,2}(M^{mu})$ and ${\rm D}({\bf \Delta}^{M^{mu}})$ as $n \to \infty$. Similarly, for any  $\varphi \in {\rm D}_{L^\infty}(\Delta^{M^{mu}})\cap \Lip (M^{mu}) $, $\varphi\geq 0$,  we can define $\varphi_n:=\chi_{a'_n} {\mathrm h}_{b'_n} (\chi_{a'_n} \varphi)$ in the same way with some $a'_n, b'_n$.

It can be checked that 
$$\tilde{\Gamma}_2^{mu}(f_n, \varphi_n) \to \tilde{\Gamma}_2^{mu}(f, \varphi)={\Gamma}_2^{mu}(f, \varphi),~~~\int |\D f_n|^2\varphi_n\,e^{-mu}\d\mm \to \int |\D f|^2\varphi\,e^{-mu}\d\mm$$
 as $n \to \infty$. Then we have $4)$ for such functions $f, \varphi$.  By an approximation using heat flow,  we can remove the assumption $\varphi \in \Lip$ ( see e.g. Proposition 3.6, \cite{GKKO-R}). We can also remove the assumption $f\in L^\infty$ by a simple truncation argument (see e.g. Theorem 4.8, \cite{EKS-O}). Then we prove $4)$ for all the required functions $f$ and $\varphi$.

Finally, it can be checked that the test functions $f, \varphi$ in $3)$ are included in the test functions in $4)$, so $4) \Longrightarrow 3)$.
\end{proof}

\subsection{Equivalent characterizations}
In this part we will prove the main results in this paper. The first theorem characterizes the $K$-convex functions on $\rcd$ space. 
Due to lack of knowledge about the regularity of weak $K$-convex functions,  we   assume a priori  that $u$ 
satisfies the following properties.

\begin{assumption}\label{assumption:u}
The function $u$ satisfies the following properties:
\begin{itemize}
\item [i)]$u\in L^1_{\rm loc}(X, \mm)$ and $u$ is lower semi-continuous,
\item [ii)]$u(x) \geq -a-b \d^2(x, x_0)$ for some $a, b \in \R$, $x_0 \in X$.

Assumptions i) and ii) ensure that the functional $\mathcal{P}_2 \ni \mu\mapsto \int u\,\d \mu$ is lower semi-continuous, and not identically $-\infty$.

\item [iii)]$\nabla u \in L^2_\loc(TM)$,
\item [iv)] there exists a unique regular Lagrangian flow associated to $-\nabla u$.
\end{itemize}
\end{assumption}

\begin{theorem}\label{mainth-1}
Let $\ms$ be a $\rcd$ metric measure space. Assume that $u$  fulfils Assumption \ref{assumption:u}.  We denote the regular Lagrangian flow associated to $-\nabla u$ by $(F_t)$.
Then the following characterizations are equivalent.
\begin{itemize}
\item [1)] $u$ is weakly $K$-convex.
\item [2)] $\nabla u$ is $K$-monotone.
\item [3)] The exponential contraction in Wasserstein distance:
\[
W_2(\mu^1_t, \mu^2_t)\leq e^{-Kt}W_2(\mu^1_0, \mu^2_0), ~~\forall t>0
\]
holds for any two absolutely continuous curves   $(\mu_t^1), (\mu_t^2)\subset (\mathcal{P}_2, W_2)$ with bounded compression,  whose velocity fields are $-\nabla u$.
\item [4)] The regular Lagrangian flow $(F_t) $ associated to $-\nabla u$ has a unique continuous representation $X$. Furthermore, the exponential contraction
\[
\d(F_t(x), F_t(y)) \leq e^{-Kt} \d(x, y)
\]
holds for any $x, y \in X$ and $t>0$.
\item [5)] For any $f\in W^{1,2}\ms$, we have $f\circ F_t \in W^{1,2}$ for any $t>0$, and 
\[
|\D (f\circ F_t) |(x) \leq e^{-Kt}|\D f|\circ F_t(x), ~~\mm-\text{a.e.}~ x\in X
\]

Furthermore, if  $u \in L^\infty_\loc \cap W^{2,2}_\loc$, then one of the above characterizations  holds
if and only if :
\item [6)] $u$ is infinitesimally $K$-convex.
\end{itemize}
\end{theorem}
\begin{proof}

\underline{$1)\Longrightarrow 2)$}:  Let $\mu_0, \mu_1 \in \mathcal{P}_2$ be any two measures with bounded densities and bounded supports. We consider the (unique) geodesic $(\mu_t)_{t\in [0,1]}$ from $\mu_0$ to $\mu_1$. From weak $K$-convexity, we know
\begin{equation}\label{eq2-th1}
U(\mu_s) \leq \frac{(1-s)}{1-t}U(\mu_t)+\frac{s-t}{1-t}U(\mu_1)-\frac{K}{2}\frac{(1-s)(s-t)}{1-t}W^2_2(\mu_0, \mu_1),~~\forall s\in [t,1],
\end{equation}
where $U(\mu)=\int u\,\d \mu$. Therefore, 
\begin{equation}\label{eq3-th1}
\frac{U(\mu_s)-U(\mu_t)}{s-t}\leq \frac{1}{1-t}\Big [ U(\mu_1)-U(\mu_t)\Big ]-\frac{K}{2}\frac{(1-s)}{1-t}W^2_2(\mu_0, \mu_1).
\end{equation}
Letting $s\downarrow t$ and $t\downarrow 0$ in \eqref{eq3-th1}, by Proposition \ref{prop-conteq}, Proposition \ref{prop:geod},   $C^1$ continuity of geodesics in Proposition \ref{prop:c1}, and lower semicontinuity of $U$  we obtain
\begin{equation}\label{eq4-th1}
-\int \la \nabla u, \nabla \varphi \ra \,\d \mu_0\leq U(\mu_1)-U(\mu_0)-\frac{K}{2}W^2_2(\mu_0, \mu_1),
\end{equation}
where  $\varphi$ is a Kantorovich potential from $\mu_0$ to $\mu_1$.

Similarly, by changing the role of $\mu_1$ and $\mu_0$ we obtain 
\begin{equation}\label{eq5-th1}
-\int \la \nabla u, \nabla \varphi^c \ra \,\d \mu_1\leq U(\mu_0)-U(\mu_1)-\frac{K}{2}W^2_2(\mu_0, \mu_1).
\end{equation}
Combining \eqref{eq4-th1} and \eqref{eq5-th1} we obtain 
\[
\int \la \nabla u, \nabla \varphi \ra \,\d \mu_0+\int \la \nabla u, \nabla \varphi^c \ra \,\d \mu_1 \geq KW^2_2(\mu_0, \mu_1).
\]
Since $\mu_0, \mu_1$ are arbitrary, we know $\nabla u$ is $K$-monotone.

\bigskip

\underline{$2)\Longrightarrow 1)$}: By an approximation argument,  it is sufficient to prove 
\[
U(\mu_{\frac12}) \leq \frac12 U(\mu_0)+\frac12U(\mu_1)-\frac{K}{8}W^2_2(\mu_0, \mu_1)
\]
for any  geodesic $(\mu_t) \subset (\mathcal{P}_2, W_2)$,  where $\mu_0, \mu_1$ have bounded densities.

From Proposition \ref{prop-1} and Proposition \ref{prop:geod} we know
\begin{eqnarray*}
U(\mu_\frac12)-U(\mu_0)&=&\int^\frac12_0 \Big (\frac{\d}{\d r} \int u \,\d \mu_r\Big )\, \d r\\
&=& \int_0^\frac12 \frac1{1-2r}\Big (\frac{\d}{\d s}\restr{s=0} \int u \,\d \mu_{r+s(1-2r)}\Big )\, \d r\\
&=& -\int^\frac12_0 \frac1{1-2r}\Big ( \int \la \nabla u, \nabla \varphi_{r, 1-r}\ra\,\d \mu_{r}\Big )\, \d r,
\end{eqnarray*}
where $\varphi_{r, 1-r}$ is a Kanrotovich potential relative to $(\mu_r, \mu_{1-r})$.

Similarly, we have
\begin{eqnarray*}
U(\mu_1)-U(\mu_\frac12)&=& \int_\frac12^1 \frac1{2r-1}\Big ( \int \la \nabla u, \nabla (\varphi_{1-r, r})^c\ra\,\d \mu_{r}\Big )\, \d r.
\end{eqnarray*}
By a change of variable, we obtain 
\[
U(\mu_1)-U(\mu_\frac12)=\int^\frac12_0 \frac1{1-2r}\Big ( \int \la \nabla u, \nabla (\varphi_{r, 1-r})^c\ra\,\d \mu_{1-r}\Big )\, \d r.
\]

Combining the results above, we obtain
\begin{eqnarray*}
&&\frac12 U(\mu_0)+\frac12U(\mu_1)-U(\mu_{\frac12})\\ &=& \frac12 \Big (U(\mu_0)-U(\mu_\frac12) \Big )+ \frac12 \Big (U(\mu_1)-U(\mu_\frac12) \Big )\\
&=&\frac12 \int^\frac12_0 \frac1{1-2r}\Big ( \int \la \nabla u, \nabla (\varphi_{1-r, r})^c\ra\,\d \mu_{r}+ \int \la \nabla u, \nabla (\varphi_{r, 1-r})^c\ra\,\d \mu_{1-r}\Big)\, \d r\\
&\geq & \frac12 \int^\frac12_0 \frac1{1-2r} K(1-2r)^2W_2^2(\mu_0, \mu_1)\,\d r\\
&=& \frac{K}8 W_2^2(\mu_0, \mu_1),
\end{eqnarray*}
which is the thesis.
\bigskip

\underline{$1)\Longrightarrow 3)$}: Let $\mu_0 \in \mathcal{P}_2(X)$ be a measure with bounded density and bounded support, $(\mu_t)$ be the  RLF associated to $-\nabla u$ starting from $\mu_0$.   Assume that $\mu_t, t\in [0, T]$ have uniformly bounded supports. We claim that $(\mu_t)$ is an ${\rm EVI}_K$-gradient flow of $U$  in the following sense:
\begin{equation}\label{th1-eq-evi}
\frac{\d }{\d t} \frac12 W^2_2(\mu_t, \nu)+\frac{K}2 W^2_2(\mu_t, \nu) \leq U(\nu)-U(\mu_t),~~~\text{for all}~t>0
\end{equation}
for any $\nu \in \mathcal{P}_2(X)$.  It is sufficient to prove \eqref{th1-eq-evi} for any $\nu$ with bounded density and compact support (cf. Proposition 2.21 \cite{AGS-M}).

By Proposition \ref{prop:derw2} we have
\begin{equation}\label{eq6-th1}
\frac{\d }{\d t} \frac12 W^2_2(\mu_t, \nu)=-\int \la \nabla u, \nabla \varphi_{t}\ra\,\d \mu_t
\end{equation}
for a.e. $t>0$, where $\varphi_{t}$ is the Kantorovich potential from $\mu_t$ to $\nu$.  From  \eqref{eq4-th1}, we know  
\begin{equation}\label{th1-eq-evi2}
-\int \la \nabla u, \nabla \varphi_t \ra \,\d \mu_t\leq U(\nu)-U(\mu_t)-\frac{K}{2}W^2_2(\mu_t, \nu),~~~~~\forall~t\geq 0.
\end{equation}
Combining \eqref{th1-eq-evi2} and \eqref{eq6-th1} we know \eqref{th1-eq-evi} holds for a.e. $t>0$. To prove the claim, it is sufficient to 
prove  the $C^1$-continuity of the function $t\mapsto W^2_2(\mu_{t}, \nu)$.
 So we  need to prove 
$$
\lmt{h}{0}\int \la \nabla u, \nabla \varphi_{t+h}\ra\,\d \mu_{t+h}=\int \la \nabla u, \nabla \varphi_{t}\ra\,\d \mu_t
$$
for any given $t$. From Lemma 2.3 in \cite{AGMR-R}, Proposition \ref{prop:c1-2} and the compactness of $\supp \nu$, we know the compactness/stability of Kantorovich potentials.  Combining with Proposition \ref{prop:c1-2}, Proposition \ref{prop:brenier} and uniform boundedness of $\supp \mu_t$,  we can prove the convergence using Lemma \ref{lemma:weakstrong}.

Let  $(\nu_t)$ be another  RLF associated to $-\nabla u$ starting from $\nu_0$,  $\nu_0$ has  bounded density and bounded support such that $\nu_t, t\in [0, T]$ have uniformly bounded supports. Then by Theorem 4.0.4 in \cite{AGS-G} we have the exponential contraction:
\begin{equation}\label{eq10-th1}
W_2(\mu_t, \nu_t)\leq e^{-Kt}W_2(\mu_0, \nu_0)
\end{equation}
for any $t\geq 0$.

For arbitrary $\mu_0, \nu_0\in \mathcal{P}_2$ with bounded densities,
  we can restrict $\mu_0, \nu_0$ on those points $x\in X$ such that  $F_t(x)\subset B_R(x_0)$ for any $t\in [0, T]$,  where  $x_0 \in X$,  $R>0$. Then we can renormalise $\mu_0, \nu_0$ and denote them by $\mu^R_0, \nu^R_0$. We  push-forward $\mu^R_0, \nu^R_0$ by $F_t$ and denote them by $(\mu^R_t), (\nu^R_t)$.   It can be seen that \eqref{eq10-th1} holds for  $(\mu^R_t), (\nu^R_t)$.
  
Letting $R \to \infty$ we know  $\mu^R_0, \nu^R_0$ converge to $\mu_0, \nu_0$  respectively in $(\mathcal{P}_2, W_2)$. From the completeness of $(\mathcal{P}_2, W_2)$, we know $(\mu^R_t), (\nu^R_t)$ converge to some $(\mu_t), (\nu_t)$. It can be seen from the uniqueness of RLF that $\mu_t=(F_t)_\sharp \mu_0$ and  $\nu_t=(F_t)_\sharp \nu_0$.  So \eqref{eq10-th1} holds for $(\mu_t), (\nu_t)$.
\bigskip

\underline{$3)\Longrightarrow 4)$}: Let $x\in X$ be an arbitrary point.  From exponential contraction,  by an approximation argument we know the flow associated to $-\nabla u$ starting at $\delta_x \in \mathcal{P}_2$ is uniquely defined.
In fact, for any $x\in X$, we can find a sequence $(\mu^n)\subset \mathcal{P}_2$ such that $\lmt{n}{\infty}W_2(\mu^n, \delta_x)=0$. From \eqref{eq10-th1} we know the flows associated to $-\nabla u$ from $\mu^n$, which is denoted by $(\mu^n_t)_t$,  converges uniformly to a  curve as $n \to \infty$. It can be seen that this limit curve is independent of the choice of $(\mu^n)_n$.
We denote this curve by $\big(U_t(x) \big )_t\subset \mathcal{P}_2(X)$. Now we claim that $U_t(x)$ supports on a single point in $X$.   Assume there exists $t_0>0$ such that $\supp U_{t_0}(x)$ has at least two  points $a, b \in X$. Let $\Pi^n\in \mathcal{P}(C([0, \infty), X))$ be the lifting of $(F_t)_\sharp \big (\frac1{\mm(B_{\frac1n}(x))} \mm\restr{B_{\frac1n}(x)}\big )$. Since RLFs are non-branching, we know there exist $\Gamma^{1,n}, \Gamma^{2,n} \in \supp \Pi^n$ with positive measures such that $\inf \{ \d(\gamma^1_{t_0}, \gamma^2_{t_0}): \gamma^1 \in \Gamma^{1,n}, \gamma^2 \in \Gamma^{2,n} \}>\frac12 \d(a, b)>0$ for $n$ large enough.   Then, by renormalization, we find two  sequences of curves $\mu^{i,n}_t:=(e_t)_\sharp\big(\frac1{\Pi^n(\Gamma^{i,n})} \Pi^n\restr{\Gamma^{i,n}} \big) , i=1,2$,   such that  $\mu^{i,n}_0 \to \delta_x$ but $\mu^{1,n}_{t_0}\neq \mu^{2,n}_{t_0}$  which contradicts to the uniqueness of $U_t(x)$. We still use $U_t(x)$ to denote this single point.

Let $x\in X$ be a point where the curve $(F_t(x))_t$ is  well-defined (i.e. $(F_t(x))_t$ is an absolutely continuous curve in $X$),  where $(F_t)$ is the RLF associated to $-\nabla u$. From the  construction procedure of $U_t$ and the uniqueness of $U_t(x)$ we know $U_t(x)=F_t(x) $.  

Therefore, we can extend $F_t$ to the whole space in the following way. For any $x\in X$,  we define $(F_t)_\sharp \delta_x=U_t(x)= \delta_{F_t(x)}$. Then we complete the proof by applying  \eqref{eq10-th1}  with $\mu_0=\delta_x, \mu_1=\delta_y$.

\bigskip

\underline{$4)\Longrightarrow 5)$}: 
Since $f\in W^{1,2}$, we know there exists a sequence $(f_n) \subset \Lip(X)$ such that $f_n \to f$ and $|\lip{f_n}| \to |\D f|$ in $L^2$. Then we have
\begin{eqnarray*}
\lmt{n}\infty\int |f\circ F_t-f_n\circ F_t|(x)^2\, \d\mm&=&\lmt{n}\infty\int |f-f_n|^2(x)\,\d(F_t)_\sharp \mm\\
&\leq& C\lmt{n}\infty\int |f-f_n|^2\,\d\mm\\
&=& 0,
\end{eqnarray*}
where we use  $(F_t)_\sharp \mm\leq C\mm$ in the second step.  Similarly, we can prove that $(|\lip{f_n}|\circ F_t)_n$ converges to $ |\D f|\circ F_t$ in $L^2$.

From the hypothesis, we know
\begin{eqnarray*}
|\lip{f_n\circ F_t}|(x)&=&\lmts{y}{x} \frac{|f_n\circ F_t(y)-f_n\circ F_t(x)|}{\d(y,x)}\\
&=& \lmts{y}{x} \frac{|f_n\circ F_t(y)-f_n\circ F_t(x)|}{\d(F_t(x), F_t(y))}\frac{\d(F_t(x), F_t(y))}{\d(y,x)}\\
&\leq&  \lmts{y}{x} \frac{|f_n\circ F_t(y)-f_n\circ F_t(x)|}{\d(F_t(x), F_t(y))} \lmts{y}{x}\frac{\d(F_t(x), F_t(y))}{\d(y,x)}\\
&\leq&  |\lip{f_n}|\circ F_t(x) e^{-Kt}.
\end{eqnarray*}

Then we obtain
\begin{eqnarray*}
\lmts{n}\infty \int |\lip{f_n\circ F_t}|^2\,\d\mm &\leq& \lmt{n}{\infty} e^{-2Kt}\int |\lip{f_n}|^2\circ F_t\,\d \mm\\
&=& e^{-2Kt}\int |\D f|^2\circ F_t\,\d \mm.
\end{eqnarray*}
Hence by definition we know $f\circ F_t \in W^{1,2}$.

Moreover, let $G$ be a weak limit of a subsequence of $(\lip{f_n\circ F_t})_n$ in $L^2$. By pointwise minimality of weak gradient, we get $|\D (f\circ F_t)| \leq G \leq e^{-Kt} |\D f|\circ F_t$ $\mm$-a.e..

\bigskip

\underline{$5)\Longrightarrow 3)$}:  The strategy used here is similar to the proofs in \cite{GH-C} and \cite{Kuwada10}, so we sketch the proof.  We just need to prove $3)$ for $\mu^1_0, \mu^2_0$ with the form $\mu^1_0=f\mm$ and $\mu^2_0=g\mm$, where $f, g$ are Lipschitz functions with bounded supports.
Now let $\varphi \in L^\infty\cap \Lip $ be a function with bounded support. We denote by $(\nu^0_r)_r$  the geodesic connecting $\mu^1_0$ and $\mu^2_0$,  and denote $(F_t)_\sharp \nu^0_r$ by $\nu^t_r$.  We also denote the velocity field of $(\nu^0_r)_r$ by $(\nabla \phi^0_r)$. 

For any $r\in [0,1], h>0$, we have
\begin{eqnarray*}
&&\Big  |{\int    Q_{r+h}(\varphi)\,\d \nu^t_{r+h}- \int    Q_r(\varphi)\,\d \nu^t_r}{}\Big |\\
&\leq &\Big | {\int    Q_{r+h}(\varphi)\,\d \nu^t_{r+h}- \int    Q_r(\varphi)\,\d \nu^t_{r+h}}{}\Big|+\Big | {\int    Q_{r}(\varphi)\,\d \nu^t_{r+h}- \int    Q_r(\varphi)\,\d \nu^t_r}{}\Big |\\
&\leq& C\int   \big|Q_{r+h}(\varphi)-     Q_r(\varphi)\big |\,\d\mm +\Big | {\int    Q_{r}(\varphi)\,\d \nu^t_{r+h}- \int    Q_r(\varphi)\,\d \nu^t_r}{}\Big |.
\end{eqnarray*}
Then we know that $r \mapsto \int    Q_r(\varphi)\,\d \nu^t_r$ is absolutely continuous, so it is differentiable almost everywhere.
Using weak Leibniz rule (cf. Lemma 4.3.4, \cite{AGS-G}) we have
\begin{eqnarray*}
&&\frac{\d}{\d r} \int    Q_r(\varphi)\,\d \nu^t_r\\
&=&\lmt{h}{0}\frac {\int    Q_{r+h}(\varphi)\,\d \nu^t_{r+h}- \int    Q_r(\varphi)\,\d \nu^t_r}{h}\\
&\leq &\lmts{h}{0}\frac {\int    Q_{r+h}(\varphi)\,\d \nu^t_{r}- \int    Q_r(\varphi)\,\d \nu^t_{r}}{h}+\lmts{h}{0}\frac {\int    Q_{r}(\varphi)\,\d \nu^t_{r}- \int    Q_r(\varphi)\,\d \nu^t_{r-h}}{h}
\end{eqnarray*}
for a.e.  $r$.

By Hamilton-Jacobi equation in Lemma \ref{lemma-hj1}, Proposition \ref{prop:c1} and dominated convergence theorem we have 
\[
\lmt{h}{0}\frac {\int    Q_{r+h}(\varphi)\,\d \nu^t_{r}- \int    Q_r(\varphi)\,\d \nu^t_{r}}{h}=\int -\frac12|\D Q_r(\varphi)|^2 \d \nu^t_r
\]
for a.e.  $r\in (0, 1)$.
From  Proposition \ref{prop:c1} we know
\[
\lmt{h}{0}\frac {\int    Q_{r}(\varphi)\,\d \nu^t_{r+h}- \int    Q_r(\varphi)\,\d \nu^t_r}{h}= \int  \la \nabla (Q_r(\varphi)\circ F_t), \nabla \phi^0_r \ra \d \nu^0_r
\]
for all $r$.

Combining with  the computations above  we obtain:
\[
\frac{\d}{\d r} \int    Q_r(\varphi)\,\d \nu^t_r \leq \int -\frac12|\D Q_r(\varphi)|^2 \d \nu^t_r+\int  \la \nabla (Q_r(\varphi)\circ F_t), \nabla \phi^0_r \ra \d \nu^0_r
\]
for a.e.  $r\in (0, 1)$.

Then we have the following estimate:
\begin{eqnarray*}
&&\int \varphi^c(y)\,\d  \mu^2_t(y)+\int  \varphi (x)\,\d \mu^1_t(x)\\
&=& \int_0^1 \frac{\d}{\d r} \int  \Big (Q_r(-\varphi) \d \nu^t_r\Big )\, \d r\\
 &\leq & \int_0^1 \int -\frac12|\D Q_r(-\varphi)|^2 \d \nu^t_r\, \d r\\
&&~~~~~~~~~~+ \int_0^1 \int  \la \nabla (Q_r(-\varphi)\circ F_t), \nabla \phi^0_r \ra \d \nu^0_r\, \d r\\
\big(\text{Young's inequality }\big)&\leq&  \int_0^1 \int -\frac12|\D Q_r(-\varphi)|^2 \d \nu^t_r\, \d r\\
&&~~~~~~~~~~+ \frac12  \int_0^1 \int  e^{2Kt}|\D(Q_r(-\varphi)\circ F_t)|^2\,\d \nu^0_r\d r \\ &+&\frac12 e^{-2Kt}\int^1_0\int |\D \phi^0_r|^2\,\d \nu^0_r\d r\\
\big(\text{hypothesis }~5)\big) &\leq& \frac12 e^{-2Kt}\int^1_0\int |\D \phi^0_r|^2\,\d \nu^0_r\d r\\
\big(\text{Proposition}~\ref{prop-conteq}+\text{Proposition}~\ref{prop:geod}\big)&= & \frac12 e^{-2Kt}W_2^2(\mu^1_0, \mu^2_0).
\end{eqnarray*}
Since $\varphi$ is arbitrary, we know 
$W_2^2(\mu^1_t, \mu^2_t) \leq e^{-2Kt}W_2^2(\mu^1_0, \mu^2_0)$.

\bigskip

\underline{$4)+5)\Longrightarrow 2)$}: Let $\mu_0, \nu_0 \in \mathcal{P}_2$ be probability measures  with compact supports and bounded  densities. We consider the RLFs $(\mu_t)_{t\in [0,T]}$ and $(\nu_t)_{t\in [0,T]}$ starting from $\mu_0, \nu_0$ respectively, where  $T>0$. From Proposition \ref{prop:rlf} we know the measures  $\mu_t, \nu_t, t\in [0, T]$ have uniformly bounded densities.    From $4)$ we know that $\mu_t, \nu_t$ have compact supports for all $ t\in [0, T]$,  and the supports of $\mu_t, \nu_t$,  $t\in [0, T]$ are uniformly bounded.

We denote by $(\theta_r)_r$ the geodesic  from $\mu_{0}$ to $\nu_{0}$, and denote  by $\nabla \phi_r$ the velocity field of $(\theta_{r})_r$. Let $\delta_r:[0,1] \mapsto [0,1]$ be a $C^1$ function (to be determined) with $\delta(i)=i, i=0,1$.  We define an interpolation  $(F_{tr })_\sharp \theta_{\delta_r}$ and denoted it by $\eta^t_r$.    

Then we  estimate $W^2_2(\mu_{0}, \nu_{t})$ using a similar method as we used in $5)\Longrightarrow 3)$.  For any $\varphi \in L^\infty\cap \Lip $  with bounded support, we have
\begin{eqnarray*}
&&\int \varphi^c(y)\,\d  \nu_{t}(y)+\int  \varphi (x)\,\d \mu_{0}(x)\\
&=&\int \varphi^c(y)\,\d  \eta^t_{1}(y)+\int  \varphi (x)\,\d \eta^t_{0}(x)\\
&=& \int_0^1 \frac{\d}{\d r} \int  \Big (Q_r(-\varphi) \circ F_{tr} \Big )\d \theta_{\delta_r}\, \d r\\
 &=& \int_0^1 \int -\frac12|\D Q_r(-\varphi)|^2\, \d \eta^t_r \d r
+ \int_0^1 \delta'_r \int  \la \nabla (Q_r(-\varphi)\circ F_{tr}), \nabla \phi_{\delta_r} \ra\, \d \theta_{\delta_r} \d r\\
 &-& t\int_0^1 \int  \la \nabla (Q_r(-\varphi), \nabla u \ra\, \d \eta^t_r \d r\\
 &=&  \int_0^1 \int -\frac12|\D \big(Q_r(-\varphi)+tu\big)|^2\, \d \eta^t_r \d r
+ \int_0^1 \delta'_r \int  \la \nabla \big ((Q_r(-\varphi)+tu)\circ F_{tr}\big ), \nabla \phi_{\delta_r} \ra\, \d \theta_{\delta_r} \d r\\
 &+& \int_0^1 \int  \frac12 t^2 |\D u|^2\, \d \eta^t_r \d r-t\int_0^1 \delta'_r \int  \la \nabla  (u\circ F_{tr} ), \nabla \phi_{\delta_r} \ra\, \d \theta_{\delta_r} \d r\\
 &\leq& \int_0^1  \frac12 \big(\delta'_r\big )^2 e^{-2Krt}\int |\D \phi_{\delta_r}|^2  \,\d \theta_{\delta_r} \d r +\int_0^1 \int  \Big[\frac12 t^2 \la \nabla (u\circ F_{tr} ), \nabla u\ra-t\delta'_r   \la \nabla  (u\circ F_{tr} ), \nabla \phi_{\delta_r} \ra \Big ]\, \d \theta_{\delta_r} \d r\\
 &:=& A(t)+tB(t).
\end{eqnarray*}

We then choose 
\[
\delta(r):=\frac{e^{2Krt}-1}{e^{2Kt}-1},
\]
so that $\delta'(r)=R_K(t)e^{2Krt}$ where 
$$
R_K(t):=\frac{2Kt}{e^{2Kt}-1}~~\text{if}~K \neq 0, ~~~~~R_0(t)=1.
$$
Then we have
\begin{eqnarray*}
A(t) &=& \int_0^1  \frac12 \big(\delta'_r\big )^2 e^{-2Krt}\int |\D \phi_{\delta_r}|^2  \,\d \theta_{\delta_r} \d r \\
&=& \frac{R_K(t)}{2}\int_0^1  \delta'_r \int |\D \phi_{\delta_r}|^2  \,\d \theta_{\delta_r} \d r\\
&=& \frac{R_K(t)}{2}\int_0^1  \int |\D \phi_{r}|^2  \,\d \theta_{r} \d r\\
&=& \frac12 R_K(t) W^2_2(\mu_0, \nu_0)
\end{eqnarray*}

It can be seen from Proposition \ref{prop:c1} and Proposition \ref{prop:c1-2}  that $B(t)$ is continuous in $t$. In fact, by direct computation we can even prove
\begin{eqnarray*}
U(\mu_0)-U(\nu_t) &=& \int_0^1 \frac{\d}{\d r} \int  \Big ((-u)\circ F_{rt} \Big )\,\d\theta_{\delta_r } \d r\\
&=&\int_0^1 \int  \Big[ t \la \nabla (u\circ F_{tr} ), \nabla u\ra-\delta'_r   \la \nabla  (u\circ F_{tr} ), \nabla \phi_{\delta_r} \ra \Big ]\, \d \theta_{\delta_r} \d r\\
& \geq & B(t).
\end{eqnarray*}

Combining the results above, we obtain
\[
W^2_2(\mu_0, \nu_t)\leq  R_K(t) W^2_2(\mu_0, \nu_0)+2tB(t).
\]
Subtracting $W_2^2(\mu_0,\nu_0)$ on both sides, and dividing $t>0$ on both sides. Letting $t \to 0$, together with the  formula $R_K(t)=1-\frac{Kt}2+o(t)$ we obtain
\[
\frac{\d^+}{\d t} W^2_2(\mu_0, \nu_t) \restr{t=0} \leq B(0) -\frac{K}{2}W^2_2(\mu_0, \nu_0).
\]
Since $t \mapsto W^2_2(\mu_0, \nu_t) $ is $C^1$ (cf. the proof of $1)\Longrightarrow 3)$), we know
\begin{equation}\label{eqB-th1}
-\int \la \nabla u, \nabla \varphi_{0,0}^c \ \ra\,\d\nu_0  \leq B(0) -\frac{K}{2}W^2_2(\mu_0, \nu_0),
\end{equation}
where $B(0)=-\int_0^1 \int      \la \nabla  u, \nabla \phi^0_r \ra \, \d \nu^0_{r} \d r$.

Using the same argument we can also prove
\begin{equation}\label{eqC-th1}
-\int \la \nabla u, \nabla \varphi_{0,0} \ \ra\,\d\mu_0  \leq C(0) -\frac{K}{2}W^2_2(\mu_0, \nu_0),
\end{equation}
where $C(0)=\int_0^1 \int      \la \nabla  u, \nabla \phi^0_{1-r} \ra \, \d \nu^0_{1-r} \d r=\int_0^1 \int      \la \nabla  u, \nabla \phi^0_{r} \ra \, \d \nu^0_{r} \d r=-B(0)$.

Combining \eqref{eqB-th1} and \eqref{eqC-th1} we obtain
\begin{equation}\label{eqD-th1}
\int \la \nabla u, \nabla \varphi_{0,0}^c \ \ra\,\d\nu_0+\int \la \nabla u, \nabla \varphi_{0,0} \ \ra\,\d\mu_0  \geq KW^2_2(\mu_0, \nu_0).
\end{equation}

Finally, by  approximation by compactly supported measures and  metric Brenier's theorem, we can prove that \eqref{eqD-th1} holds 
for all $\mu_0, \nu_0$ with bounded supports and bounded densities,  so $\nabla u$ is $K$-monotone.

\bigskip

\underline{$6) \Longleftrightarrow 1)$}: This is a direct consequence of Theorem \ref{mainth-0}.
\end{proof}

\bigskip

 \begin{remark}
Let $f$ be  a smooth function $f$ on a Riemannian manifold $(M, g)$,  and  $(\gamma_t)$ be a smooth curve. We know the function $ t \to f(\gamma_t)$ is smooth and
\[
\frac{\d^2}{\d t^2} f(\gamma_t)=\H_f(\gamma'_t, \gamma'_t)+\la \nabla_{\gamma'_t}\gamma'_t, \nabla f \ra.
\]
In particular, if $(\gamma_t)$ is a geodesic, we know  $\nabla_{\gamma'_t}\gamma'_t=0$. Then we obtain
\[
\frac{\d^2}{\d t^2} f(\gamma_t)=\H_f(\gamma'_t, \gamma'_t).
\]
So the second order derivative along geodesic  characterizes the convexity of a function $f$. 
 
 On $\rcdkn$ spaces, we can use the second  order differentiation formula developed by Gigli-Tamanini \cite{GT-S} to study the convexity of $H^{2,2}$ functions. 
 However, it is still unknown to us whether we have such formula in $\rcd$ case or not. 
\end{remark}

\bigskip

\begin{theorem}\label{mainth-2}
Let $M:=\ms$ be a $\rcd$  space, ${\bf b} \in L^2_\loc(TM)$. We assume there exits a unique  regular Lagrangian flow associated to $-{\bf b}$,   and denote it by $(F_t)$. Then the following descriptions are equivalent.

\begin{itemize}
\item [1)] ${\bf b}$ is $K$-monotone.
\item [2)] The exponential contraction in Wasserstein distance:
\[
W_2(\mu^1_t, \mu^2_t)\leq e^{-Kt}W_2(\mu^1_0, \mu^2_0), ~~\forall t>0
\]
holds for any two curves  $(\mu_t^1), (\mu_t^2)$ with velocity field $-{\bf b}$.

\item [3)] The regular Lagrangian flow $(F_t) $ associated to $-{\bf b}$ has a unique continuous representation $X$, so that $F_t(x)$ is uniquely defined everywhere.  Furthermore  the exponential contraction
\[
\d(F_t(x), F_t(y)) \leq e^{-Kt} \d(x, y)
\]
holds for any $x, y \in X$ and $t>0$.
\item [4)] For any $f\in W^{1,2}\ms$, we have $f\circ F_t \in W^{1,2}$ and
\[
|\D (f\circ F_t) |(x) \leq e^{-Kt}|\D f|\circ F_t(x), ~~\mm-\text{a.e.} ~x\in X.
\]
\end{itemize}

\end{theorem}
\begin{proof} We can prove $ 2) \Longrightarrow 3)\Longrightarrow 4)\Longrightarrow 2)$ and  $4) \Longrightarrow 1)$
 in the same ways as in the proof of Theorem \ref{mainth-1}.

\underline{$1)\Longrightarrow 2)$}: Let $\mu_0, \nu_0 \in \mathcal{P}_2$ be two measures with bounded supports and bounded densities, $(\mu_t), (\nu_t)$ be the solutions to the continuity equation with velocity field $-{\bf b}$,   with initial  datum $\mu_0$ and $\nu_0$ respectively.  By definition, we know that $\mu_t, \nu_t$ have bounded densities  for any $t>0$.  Fix  $T>0$, we denote the lifting of $(\mu_t)_{t\in [0,T]}$ by
$\Pi\in \mathcal{P}({\rm AC}([0,T],X))$. Let $\Gamma \subset {\rm AC}([0,T],X)$ be the support of $\Pi$. For any $\epsilon >0$,  we can find $\Gamma_\epsilon \subset \Gamma$
which is   compact in ${\rm C}([0,T],X)$ such that  $\Pi(\Gamma\setminus \Gamma_\epsilon) <\epsilon$,  and $\{\gamma_t: \gamma\in \Gamma_\epsilon, t\in [0,1]\} \subset B_{R}(x_0)$ for some $x_0 \in X$ and $R> \frac1{\epsilon}$.  Then we define 
$$
\mu^\epsilon_t:= (e_t)_\sharp \Big (\frac1{\Pi(\Gamma_\epsilon)}\Pi\restr{\Gamma_\epsilon}\Big ),~~~~\epsilon>0,~~t\in [0,T].
$$
It can be seen that $\supp \mu^\epsilon_t =e_t(\Gamma_\epsilon)$ is compact for any $t\in [0, T]$ and 
$$
\lmt{\epsilon}{0}W_2(\mu_0, \mu_0^\epsilon)=0.
$$
So without loss of generality we can assume that $\mu_t, \nu_t$ support on compact sets for any $t\in [0,T]$. 
Furthermore, we may also assume that  $\mu_t, \nu_t$ have  uniformly bounded supports for  $t\in [0, T]$. 

 Then for any $s\geq 0$, by Proposition \ref{prop:derw2} we have
\begin{equation}\label{eq6-th2}
\frac{\d }{\d t} \frac12 W^2_2(\mu_t, \nu_s)=-\int \la {\bf b}, \nabla \varphi_{t,s}\ra\,\d \mu_t
\end{equation}
for a.e. $t>0$, where $\varphi_{t,s}$ is a Kantorovich potential from $\mu_t$ to $\nu_s$. 
Similarly, fix a $t$ we know
\begin{equation}\label{eq7-th2}
\frac{\d }{\d s} \frac12 W^2_2(\mu_t, \nu_s)=-\int \la {\bf b}, \nabla \phi_{s,t}\ra\,\d \nu_s
\end{equation}
for a.e. $s>0$, where $\phi_{s,t}$ is the Kantorovich potential from $\nu_s$ to $\mu_t$.

Now we claim that  $t\mapsto -\int \la {\bf b}, \nabla \varphi_{t,s}\ra\,\d \mu_t$ is continuous for any $s$.  We just need to prove 
$$
\lmt{h}{0}\int \la {\bf b}, \nabla \varphi_{t+h,s}\ra\,\d \mu_{t+h}=\int \la {\bf b}, \nabla \varphi_{t,s}\ra\,\d \mu_t
$$
for any given $t$.

By Proposition \ref{prop:c1-2} and the compactness assumption on $\supp \nu_s$,  we can apply Lemma 2.3 in \cite{AGMR-R} to obtain the compactness of Kantorovich potentials. Combining with Proposition \ref{prop:c1-2}   we know the convergence from Lemma \ref{lemma:weakstrong}.

Similarly, we can prove  that $s \mapsto \int \la {\bf b}, \nabla \phi_{s,t}\ra\,\d \nu_s$ is continuous. Therefore we know \eqref{eq6-th2} and \eqref{eq7-th2} hold for all $t$ and  $s$ respectively. Then we have
\begin{equation}\label{eq8-th2}
\frac{\d }{\d s} \frac12 W^2_2(\mu_t, \nu_s)\restr{s=t}=-\int \la {\bf b}, \nabla \phi_{t,t}\ra\,\d \nu_t
\end{equation}
and
\begin{equation}\label{eq9-th2}
\frac{\d }{\d r} \frac12 W^2_2(\mu_r, \nu_t)\restr{r=t}=-\int \la {\bf b}, \nabla \varphi_{t,t}\ra\,\d \mu_t.
\end{equation}
Furthermore, we know $t \mapsto W^2_2(\mu_t, \nu_t)$ is differentiable for a.e. $t\in [0, T]$. Using the formula in Lemma 4.3.4, \cite{AGS-G} we have
\begin{eqnarray*}
\frac{\d }{\d t} \frac12 W^2_2(\mu_t, \nu_t)&\leq&\frac{\d }{\d r} \frac12 W^2_2(\mu_r, \nu_t)\restr{r=t}+\frac{\d }{\d s} \frac12 W^2_2(\mu_t, \nu_s)\restr{s=t}\\
&=&- \int \la {\bf b}, \nabla \varphi_{t,t}\ra\,\d \mu_t-\int \la {\bf b}, \nabla \phi_{t,t}\ra\,\d \nu_t
\end{eqnarray*}
for a.e. $t\in [0, T]$.

From the definition of $K$-monotonicity we know
\[
\frac{\d }{\d t} \frac12 W^2_2(\mu_t, \nu_t) \leq -{K}W^2_2(\mu_t, \nu_t)
\]
for a.e. $t\in [0, T]$.
Finally,  by Gr\"onwall's inequality  we  obtain the exponential contraction 
\begin{equation}\label{eq10-th2}
W_2(\mu_t, \nu_t)\leq e^{-Kt}W_2(\mu_0, \nu_0)
\end{equation}
for any $t\in [0, T]$.

\end{proof}

\begin{remark}\label{remark:locallip}
One would ask if the  infinitesimal $K$-monotonicity of ${\bf b}$ is equivalent to the characterizations in the Theorem \ref{mainth-2}.  In general,  the methods in Bakry-\'Emery theory can not be used here. For example, we even do not know the Sobolev regularity of $f\circ F_t$ when $f\in W^{1,2}$ (c.f. Lemma \ref{lemma-rlf-1}).  But in some special situations, we can achieve this goal.

Case 1. When ${\bf b}$ is a harmonic vector field on a ${\rm RCD}(0, N)$ space,  it is proved by Gigli-Rigoni \cite{GR-R} that $f\circ F_t \in {\rm TestF}$  when $f\in {\rm TestF}$, and  $F_t$ induces an isometry. 

Case 2.   On $\rcdkn$ spaces, using the second  order differentiation formula developed by Gigli-Tamanini  \cite{GT-S} we can easily prove that  infinitesimal $K$-monotonicity is equivalent to $K$-monotonicity.
\end{remark}

\bigskip

At the end of this section, we show that the $K$-monotonicity is stable under  measured Gromov-Hausdorff convergence. For simplicity, we adopt the notions from \cite{AST-W} without further explanation.  Without loss of generality, we say that $\rcd$ spaces $M_n:=(X, \d, \mm_n)$ converges to $M:=\ms$ in measured Gromov-Hausdorff topology if $\mm_n \to \mm$ weakly. 

We define the countable class
$$
\mathcal{H}_{\mathbb{Q}^+}\mathcal{A}_{\rm bs}:=\Big\{ \mathcal{H}_t f:f\in \mathcal{A}_{\rm bs}, t\in \mathbb{Q}^+\Big\}\subset \Lip \cap L^\infty,
$$
where $\mathcal{A}_{\rm bs}$ is a sub-algebra of $\mathcal{A}$ consisting of functions with  bounded support,  $\mathcal{A} \subset \Lip_{\rm b}(X)$ is the smallest set containing   
$$
\min \Big\{\d(\cdot, x), k\Big\}~~~k\in \mathbb{Q}\cap [0,\infty],~x\in D, D~\text{is dense in}~X.
$$
which is a $\mathbb{Q}$-vector space and is stable under products and lattice operations.  It can be seen (cf. \cite{AST-W}) that $\mathcal{H}_{\mathbb{Q}^+}\mathcal{A}_{\rm bs}$ is dense in $W^{1,2}$.

\begin{corollary}[Stability of $K$-monotonicity]
Let ${\bf b}_n \in W^{1,2}_C(TM_n), n\in \N$ be a sequence of velocity fields with $\sup_n\|{\bf b}_n\|_{L^2(X,\mm_n)}<\infty$ and $\sup_n \|{\rm div}{\bf b}_n\|_{L^\infty(X, \mm_n)} <\infty$. Assume  that $({\bf b}_n)_{ n\in \N}$ are  $K$-monotone, $\la {\bf b}_n, \nabla f\ra\mm_n \to \la {\bf b}, \nabla f\ra\mm$ as measures for all $f\in \mathcal{H}_{\mathbb{Q}^+}\mathcal{A}_{\rm bs}$, and
\[
\lmts{n}{\infty}\int |{\bf b}_n|^2\,\d\mm_n \leq \int |{\bf b}|^2\,\d\mm.
\]  Then  ${\bf b}$ is $K$-monotone. 
 
\end{corollary}
\begin{proof}
From Theorem 8.2  \cite{AST-W}, we know the solutions to the continuity equation with velocity field ${\bf b}_n$ converges  (in measure) to the one with velocity field ${\bf b}$. We then apply $2)$ of Theorem \ref{mainth-2} with ${\bf b}_n$.  By lower-semicontinuity of Wasserstein distance w.r.t weak topology, we know that $K$-monotonicity of ${\bf b}_n$ implies $K$-monotonicity of ${\bf b}$.
\end{proof}

\section{Applications}

In this section, we apply  Theorem \ref{mainth-1} with two special functions.  Our aim is not  to  give complete proofs to the rigidity theorems which have  already been perfectly solved, but to present how to use our result to connect the differential structure and metric structure of metric measure spaces.

\bigskip

\underline{Example 1: Splitting}

\begin{theorem}  Let $\ms$ be a $\rcd$  metric measure space. If  $\H_u=0$, and $|\D u|=1$, then there exists a metric space $Y$ such that $X$ is isometric to $Y \times \R$.
\end{theorem}
 
 \begin{proof}
    From Proposition \ref{prop:rlf-2} we know that  the regular Lagrangian flows associated to  $\nabla u$ and $-\nabla u$ exist, which  are denoted by  $(F^+_t)_{t\geq 0}$ and $({F}^-_t)_{t \geq 0}$ respectively.  By uniqueness of the RLF we get $F^+_t ({F}^-_s (x))=F^-_s(F^+_t(x))= F^{{\rm sign}(t-s)}_{|t-s|}$, where ${\rm sign} (t-s)$ is $``+"$ if  $t-s\geq 0$ and is $``-"$ if $t-s<0$.  We define
\begin{equation}
F_t(x):=
\left\{
\begin{array}{ll}
F^+_t(x)~~~~~t\geq 0,\\
F^-_t(x)~~~~~ t<0.
\end{array}\right.
\end{equation}

Since $|\D f|=1$ we know $\Lip(f)=1$ from Sobolev-to-Lipschitz property. Then we can apply  Theorem \ref{mainth-0} and Theorem \ref{mainth-1}
with infinitesimally  $0$-convex functions $ u$ and $- u$. From $4)$ of Theorem \ref{mainth-1} we know
\[
\d(F_t (x), F_{t}(y)) \leq \d(x, y)
\]
for any $x, y \in X$,  $t\in \R$. So we have
\[
\d(F_t (x), F_{t}(y)) \leq \d(x, y)=\d (F_{-t}(F_t(x)), F_{-t}(F_t(y)))  \leq \d(F_t (x), F_{t}(y)) 
\]
for any $x, y \in X$,   $t\in \R$. Hence $\d(F_t(x), F_t(y))=\d(x ,y)$ for any $x, y \in X$,   $t\in \R$.
Therefore $F_t$ induces an isometry between $u^{-1}(0)$ and $u^{-1}(t)$. Combining with the fact that $|\dot{F}_t|(x)=1$,  we know  $F_t$ induces a translation on the fibre $(F_t(x_0))_t$   for any $x_0 \in u^{-1}(0)$. It can also be checked that $u^{-1}(0)$ is totally geodesic.

Finally, it can be proved that (see Section 6, \cite{G-S} for details) the map $\Phi:  \R \times u^{-1}(0)\ni (t, x)    \mapsto F_t(x) \in X$ induces an isometry between the Sobolev spaces $W^{1,2}(\Phi^{-1}(X))$ and $W^{1,2}(\R \times u^{-1}(0))$.  Then from Sobolev-to-Lipschitz property we know that  $\Phi$ is an isometry.

 \end{proof}

\begin{remark}
In ``splitting theorem"  for ${\rm RCD}(0, N)$ spaces (cf. \cite{Cheeger-Gromoll-splitting} and  \cite{G-S}), the target function $u$ is the Buseman function associated with a  line, which is a harmonic function. From Corollary \ref{coro:hessianbound} we know $\H_u=0$. In ``spectral gap rigidity  theorem"  for $\rcd$ spaces with $k\geq  0$ (cf. \cite{GKKO-R}),  the target function $u$ is a solution to the equation $\Delta u=-ku$, by Corollary \ref{coro:hessianbound} we also have $\H_u=0$. 
\end{remark}

\bigskip

 \underline{Example 2: Volume cone  implies metric cone}

\begin{theorem} \label{theorem-cone} Let $\ms$ be a ${\rm RCD}(0, N)$ space with $\mm \ll \mathcal{H}^N$.  If  $\Delta u=N$,  $|\D u|^2=2u$ and $u\leq C\d^2(\cdot, O)$ for some $O\in X$, $C>0$, then  $(X, \d)$ admits a warped product-like structure.
\end{theorem}

\begin{proof}
Since  $\mm \ll \mathcal{H}^N$,  from the rectifiability theorem (cf.  \cite{MN-S, GP-B, KM-O})  we  know  ${\dim}_{\rm loc} =  N$ is a constant. Then from Proposition \ref{prop:finitedim}, we know $u\in W^{1,2}_{C, \loc}$ and
$\tr \H_u(x)=\Delta u (x)$ $\mm$-a.e. $x \in X$. Hence $\Delta $ is a local operator which can be written in local coordinate.

Since $\Delta u=N$, by Proposition \ref{prop:bochnerimprove} we have
\begin{equation}\label{app-eq1}
N=\Delta u=\frac12 \Delta |\D u|^2-\la \nabla u, \nabla \Delta u\ra \geq  |\H_u|_{\rm HS}^2,~~\mm-\text{a.e.} .
\end{equation}

By Cauchy inequality  and  the fact that ${\dim}_{\rm loc} =  N$ we obtain
\[
 |\H_u|_{\rm HS}^2 \geq \frac1N(\tr \H_u)^2=\frac1N(\Delta u)^2= N.
\]
Combining with \eqref{app-eq1} we know  $\H_u={\rm Id}_N$.

Then we consider the regular Lagrangian flow associated to  $\nabla u$ and $-\nabla u$, which  are denoted by $(F^+_t)_{t\geq 0}$ and $({F}^-_t)_{t \geq 0}$ respectively. We can also construct $F_t$ as we did  in the first example. We know both
\[
\d(F_t (x), F_{t}(y)) \leq e^{-Nt} \d(x, y)
\]
for any $x, y \in X$,  $t>0$, and 
\[
\d(F_t (x), F_{t}(y)) \leq e^{Nt} \d(x, y)
\]
for any $x, y \in X$,  $t <0$.

Therefore, for any $t>0$  we have
\[
\d(F_t (x), F_{t}(y)) \leq e^{-Nt}\d(x, y)=e^{-Nt}\d \big (F_{-t}(F_t(x)), F_{-t}(F_t(y))\big )  \leq e^{-Nt} e^{Nt} \d(F_t (x), F_{t}(y)).
\]

Hence  $\d(F_t(x), F_t(y))=e^{-Nt}\d(x ,y)$, and $(X, \d)$ admits a warped product-like structure.

\end{proof}

\begin{remark}
In ``volume cone implies metric cone theorem" (cf. \cite{CC-L} and \cite{DPG-F}), the target function $u$ is the squared distance function $\frac12\d^2(\cdot, {\rm O})$ where ${\rm O}$ is a fixed point. Then we know $|\D u|^2=2u=\d^2 (\cdot, {\rm O})$.  From Theorem \ref{theorem-cone} above we know there exists a scaling $F_t(\cdot)$, and $|\dot{F}_t|(x)=|\D u|\circ F_t(x)=\d (F_t(x), {\rm O})$.  By studying the Sobolev space of warped product space (cf. \cite{DPG-F, GH-S}),  we can prove  that $(X, \d)$ admits a cone structure, and the point $O$ is exactly the apex.
\end{remark}

\def\cprime{$'$}

\bigskip

Bang-Xian Han, Institute for applied mathematics, University of Bonn

 Endenicher Allee 60 , D-53115 Bonn, Germany
 
{ Email}: han@iam.uni-bonn.de
\end{document}